\documentclass[a4paper, 11pt, reqno]{amsart}

\usepackage[latin1]{inputenc}
\usepackage[T1]{fontenc}
\usepackage[english]{babel}
\usepackage{latexsym,amsmath,amssymb,amsfonts}
\usepackage{dsfont}
\usepackage{color}
\usepackage{graphicx}
\usepackage{float}
\usepackage{esint}
\usepackage{bm}
\usepackage{mathtools}
\usepackage[top=1in, bottom=1.25in, left=1in, right=1in]{geometry}
\usepackage{enumitem}
\usepackage{hyperref}
\usepackage{cleveref}
\crefname{subsection}{subsection}{subsections}
\usepackage{bbm}
\usepackage{tikz}                        
\usepackage{pgfplots}    
\usepackage{soul}         
\usepackage{subfiles}
 \usepackage{textcomp} 

\hypersetup{colorlinks = true, linkcolor = blue, citecolor = red}

\theoremstyle{plain}
\begingroup
\newtheorem{theorem}{Theorem}[section]
\newtheorem{lemma}[theorem]{Lemma}
\newtheorem{proposition}[theorem]{Proposition}

\endgroup
\theoremstyle{definition}
\begingroup
\newtheorem{definition}[theorem]{Definition}
\newtheorem{remark}[theorem]{Remark}

\endgroup
\theoremstyle{remark}

\newcommand{\N}{\mathbb N}

\newcommand{\R}{\mathbb R}

\newcommand{\tr}{{\rm Tr\,}}
\renewcommand{\div}{{\rm div}}
\newcommand{\dist}{{\rm dist}}

\newcommand{\Lip}{{\rm Lip}}

\newcommand{\hno}{{\mathcal H}^{N-1}}

\newcommand{\e}{\varepsilon}

\renewcommand{\o}{\Omega}

\def\ca{\mathbbmss{1}}

\newcommand{\f}{\mathcal{F}}

\newcommand{\h}{\mathcal{H}}

\newcommand{\average}{{\mathchoice {\kern1ex\vcenter{\hrule height.4pt
width 6pt
depth0pt} \kern-9.7pt} {\kern1ex\vcenter{\hrule height.4pt width 4.3pt
depth0pt}
\kern-7pt} {} {} }}

\DeclareMathOperator*{\essinf}{\mathrm{essinf}}

\allowdisplaybreaks

\title[Relaxation of bulk and contact energies]{On the relaxation of functionals with contact terms on non-smooth domains}
\author{Riccardo Cristoferi, Giovanni Gravina}
\address{Radboud University \\ Department of Mathematics \\ Nijmegen \\ Netherlands}
\email{riccardo.cristoferi@ru.nl}
\address{Department of Mathematical Analysis\\Faculty of Mathematics and Physics\\ 
Charles University\\Prague\\ Czech Republic}
\email {gravina@karlin.mff.cuni.cz} 
\keywords{Lower semicontinuity, relaxation, bulk and surface energies}
\subjclass[2010]{49J45, 49Q20, 26B30}
\date{\today}

\begin{document}

\begin{abstract}
We provide the integral representation formula for the relaxation in $BV(\o; \R^M)$ with respect to strong convergence in $L^1(\o; \R^M)$ of a functional with a boundary contact energy term. This characterization is valid for a large class of surface energy densities, and for domains satisfying mild regularity assumptions. Motivated by some classical examples where lower semicontinuity fails, we analyze the extent to which the geometry of the set enters the relaxation procedure.
\end{abstract}

\maketitle
\tableofcontents

\section{Introduction}
Contact energies play an important role in numerous physical and industrial applications, where boundary effects are modeled by a surface integral of the form
\begin{equation}
\label{tau-prot}
\int_{\partial \o}\tau(x, \tr u(x))\,d\hno.
\end{equation}
An important example can be found in the context of the van der Waals--Cahn--Hilliard theory of liquid-liquid phase transitions (see \cite{Cahn-CPW} and \cite{Gur}), where it is customary to consider a contact term as in \eqref{tau-prot} together with a competing bulk energy. A prototype for the energy functionals studied in this case (see \cite{MR921549}) is given by
\begin{equation}
\label{F-modica}
\f(u) \coloneqq \sigma |Du|(\o) + \int_{\partial \o} \tau(x, \tr u(x))\,d\hno.
\end{equation}
Here $\o$ is a bounded open subset of $\R^N$ with Lipschitz continuous boundary, $\tr$ denotes the trace operator on $\partial \o$, and $\hno$ is the $N - 1$-dimensional Hausdorff measure. Moreover, the phase variable $u \in BV(\o)$ represents the density of the fluid (see \Cref{sec-BV}), $\sigma$ is a positive constant, and $\tau \colon \partial \o \times \R \to \R$ is a given function that encodes the energetic interaction per unit area with the boundary of the container $\o$. An energy of this form was considered by Modica (see Proposition 1.2 in \cite{MR921549}), who proved that the functional $\f$ is lower semicontinuous in $BV(\o)$ with respect to the strong topology of $L^1(\o)$, provided the following assumptions are satisfied: 
\begin{enumerate}[label=(H.\arabic*), ref=H.\arabic*]
\item \label{H1} the domain $\o$ is of class $C^2$; 
\item \label{H2} the density function $\tau$ satisfies 
\[
|\tau(x, p) - \tau(x, q)| \le \sigma |p - q|
\]
for all $x \in \partial \o$ and for all $p,q \in \R$.
\end{enumerate}
Additionally, Modica showed that $\f$ may fail to be lower semicontinuous if $\o$ is only a Lipschitz domain, or if the Lipschitz constant of $\tau$ is strictly larger than $\sigma$. This was accomplished with the following two strikingly simple examples (see Remark 1.3 in \cite{MR921549}).
\begin{enumerate}[label=(E.\arabic*), ref=E.\arabic*]
\item \label{E1} Let $\o \coloneqq (0,1)^2$, $\sigma = 1$, and $\tau(x, p) \coloneqq \lambda p$ with $\lambda < - \sqrt{2}/2$. For $n \in \N$, let $u_n$ be defined via 
\[
u_n(x_1, x_2) \coloneqq
\left\{
\arraycolsep=5pt\def\arraystretch{1.5}
\begin{array}{ll}
0 & \text{ if } x_1 + x_2 \ge \frac{1}{n}, \\
n & \text{ if } x_1 + x_2 < \frac{1}{n}.
\end{array}
\right.
\]
Then, as one can readily check, $u_n \to 0$ in $L^1(\o)$ and $\f(u_n) = \sqrt{2} - 2\lambda < 0 = \f(0)$.
\item \label{E2} Let $\o \coloneqq \{x \in \R^2 : |x| < 1\}$, $\sigma = 1$, and $\tau(x, p) \coloneqq \lambda|p|$ with $\lambda > 1$. For $n \in \N$, let $u_n$ be defined via 
\[
u_n(x) \coloneqq \min \{|x|, (n - 1)(1 - |x|)\}.
\]
Then $u_n \to u$ in $L^1(\o)$ where $u(x) \coloneqq |x|$, and by means of a direct computation we see that $\f(u_n) - \f(u) \to 2(1 - \lambda) < 0$.
\end{enumerate}
Motivated by these observations, the aim of this paper is to obtain an integral representation formula for the lower semicontinuous envelope of $\f$ when the hypothesis in (\ref{H1}) and (\ref{H2}) are significantly weakened. In particular, we show that for a large class of domains (see \Cref{a-C1}), the roughness of the set is reflected in a growth condition for $\tau$, but poses no additional restriction on the regularity than the one required in (\ref{H2}). Thus, our work provides with the precise understanding of the interaction between the regularity of the domain and the assumptions on the density $\tau$ in the relaxation procedure. We refer to \Cref{thm:main} for the exact statement of our results. 

Regarding \eqref{H1}, let us remark that while Modica's result is stated for $C^1$ domains, to the best of our understanding, the proof presented actually requires the set to be of class $C^2$ in order to show lower semicontinuity along sequences which are not uniformly bounded in $BV(\o)$. Further insight is provided by comparing the proof of \Cref{liminf-C2} with the first step in the proof of \Cref{liminf-C1pw}. \\

It is worth noting that the study of conditions ensuring the lower semicontinuity of energy functionals with contact terms is of particular interest also in capillarity problems (for an overview on the mathematical study of capillarity phenomena see, for instance, \cite{MR0484001}). Here we only mention the contributions that are closer to the spirit of this work. In the classical paper \cite{MR336507}, Emmer considered the energy functional
\begin{equation}
\label{C-par}
\mathcal{C}(u) \coloneqq \int_\o \sqrt{1 + |Du|^2} + \int_\o u^2(x) \,dx + \nu \int_{\partial\o} \tr u(x) \,d\hno,
\end{equation}
defined for $u \in BV(\o)$, and established existence of a solution to the minimization problem for $\mathcal{C}$ under the condition that
\begin{equation}
\label{Emmer}
|\nu| < \frac{1}{\sqrt{1 + L_{\partial \o}^2}},
\end{equation}
where $L_{\partial \o}$ is the Lipschitz constant of $\o$. Notice that the first bulk contribution on the right-hand side of \eqref{C-par} is the classical non parametric area functional, which is defined via
\begin{multline*}
\int_{\o} \sqrt{1 + |Du|^2} \coloneqq \sup \bigg\{\int_{\o} (\varphi_0(x) + u(x) \div \varphi(x))\,dx : \\ (\varphi_0, \varphi) \in C^{\infty}_c(\o; \R^{N + 1}) \text{ and } \sum_{i = 0}^N\varphi_i(x)^2 \le 1 \bigg\}.
\end{multline*}
Furthermore, we recall that for $N = 2$ minimizers of the energy $\mathcal{C}$ correspond to capillary surfaces meeting the boundary of the container at an angle $\theta$, determined by the relation $\nu = - \cos \theta$. 

As previously remarked by Finn and Gerhardt (see \cite{MR434137}), the condition identified by Emmer is rather restrictive and forbids the treatment of several cases where the existence of a solution can be predicted on the basis of physical arguments. In particular, while \eqref{Emmer} correctly reveals that corner singularities on the boundary of the domain may pose an obstruction to the existence of solutions, one expects a qualitatively different result for small and large included angles, that is, depending on whether the tip of the corner points outside or inside the domain, respectively. This is the central issue addressed in \cite{MR434137}, where the authors proved existence of minimizers for Lipschitz domains satisfying an interior ball condition. A refinement of this result was obtained by Tamanini in \cite{MR483992}. Comparable results to those of Finn and Gerhardt were also obtained by Giusti in \cite{MR482506}, where the author extended the study of the capillarity problem to the case of mixed boundary conditions on two relatively open subsets of $\partial \o$. 

Finally, the parametric case, namely the study of the functional
\[
\mathcal{C}_P(E) \coloneqq (1-\lambda)|D\ca_E|(\o) + \lambda |D\ca_E|(\R^N) + \int_E \rho(x) dx,
\]
where $E \subset \R^N$ is a set of finite perimeter, $\o$ is a Lipschitz domain, and $\rho \in L^1(\o)$, was treated by Massari and Pepe in \cite{MR388973} (for the case of three fluids, see \cite{MR761749}). In this framework the existence of a solution of the volume constrained minimization problem for $\mathcal{C}_P$ was established under the assumption that $\lambda \in [0,1]$. This condition is in accordance with \eqref{H2}. \\

More recently, Fonseca and Leoni in \cite{MR1656999} considered the energy functional 
\[
\mathcal{G}(u) \coloneqq \int_{\o} h(x, u(x), \nabla u(x))\,dx + \int_{\partial \o}\tau(x, \tr u(x))\,d\hno,
\]
defined for vector-valued functions $u \in W^{1,1}(\o;\R^M)$, and proved an integral representation formula for its relaxation in $BV(\o; \R^M)$. The main assumptions required on $h$ are that $h(x, u, \cdot)$ is quasiconvex and satisfies a linear growth condition of the form 
\[
g(x, p)|\xi| \le h(x, p, \xi) \le Cg(x, p)(1 + |\xi|)
\]
for every $x \in \o$, $p \in \R^M$, and $\xi \in \R^{M \times N}$. Furthermore, $\tau \in C(\overline{\o} \times \R^M) \cap C^1(\o \times \R^M)$ is such that
\begin{equation}
\label{tau-FL}
|\nabla_p \tau(x, p)| \le g(x, p)
\end{equation}
for all $x \in \o$ and all $p \in \R^M$. Notice that in this general framework condition \eqref{tau-FL} plays the role of \eqref{H2}. Additionally, their proof hinges on an a Gauss--Green formula (see Lemma 1.2 in \cite{MR1656999}; see also \Cref{FL-lem} below) which  is shown to hold for regular domains with boundary of class $C^2$, and allows to rewrite the boundary term as a bulk contribution. The newly obtained problem can be then treated with the techniques developed in \cite{MR1177778} and \cite{MR1218685} (see also \cite{MR1183605}). 

We also mention here that in \cite{MR1885119}, Bouchitt\'e, Fonseca, and Mascarenhas obtained an integral representation formula for the relaxation in $BV$ of a functional which comprises of a general bulk term and an interfacial energy on a fixed hypersurface $\Sigma \subset \overline{\o}$. The surface densities considered in \cite{MR1885119} are nonnegative and satisfy a certain growth condition, in addition to mild regularity assumptions. Moreover, these are allowed to take the value $+ \infty$ to include in the theory the treatment of variational problems with constraints. \\

To the best of our knowledge, what is missing in the literature is a study of the relaxation of the functional $\f$ (see \eqref{F-modica}) when the contact energy $\tau$ is possibly unbounded from below and fails to satisfy \eqref{H2}. This issue is addressed in the present paper. In addition, as illustrated by Modica's example \eqref{E1} and by the several conditions suggested for the study of capillary surfaces, for the energies we consider, various local properties of $\partial \o$ must also affect the relaxation procedure. Prior to this work, the interaction of these effects has not been investigated in a framework where lower semicontinuity fails. 

In our main result (see \Cref{thm:main}), we consider the case where the boundary of the domain $\o$ is almost of class $C^1$, i.e., of class $C^1$ outside of a closed subset of $\hno$ measure zero, and the contact energy $\tau$ is only Carath\'{e}odory (or even a normal integrand, see \Cref{sec-normal}). We then obtain an integral representation formula for the effective energy $\overline{\f}$, that is, the relaxed functional associated to $\f$, under the assumption that $\tau$ satisfies a lower bound which encodes a geometric restriction. This condition can be seen as a natural generalization of that identified by Giusti in \cite{MR482506}. 



\subsection{Statement of the main result}

Let $\o$ be a bounded open subset of $\R^N$, $N \ge 2$, with Lipschitz continuous boundary. Furthermore, let $\sigma$ be a positive constant and $\tau \colon \partial \o \times \R^M \to [-\infty, \infty)$ a given Carath\'eodory function, $M \ge 1$. Throughout the paper, we assume that there are two nonnegative functions $c, L$ such that $c \in L^1(\partial \o)$, $L$ is continuous, and 
\begin{equation}
\label{tau-LB}
\tau(x, p) \ge - c(x) - L(x)|p|
\end{equation}
for $\hno$-a.e.\@ $x \in \partial \o$ and for all $p \in \R^M$. We then define 
\begin{equation}
\label{F}
\f(u) \coloneqq
\left\{
\begin{array}{ll}
\displaystyle \sigma \int_{\o}|\nabla u(x)|\,dx + \int_{\partial \o} \tau(x, \tr u(x))\, d\hno & \text{ if } u \in W^{1,1}(\o; \R^M), \\
&\\
\infty & \text{ otherwise in } L^1(\o; \R^M).
\end{array}
\right.
\end{equation}
Here and in the following we use $\tr$ to denote the trace operator on $\partial \o$. Before we proceed, let us remark that $\f$ is well defined and furthermore $\f \colon L^1(\o; \R^M) \to (-\infty, \infty]$. Indeed, since the trace space of $W^{1,1}(\o; \R^M)$ can be identified with $L^1(\partial \o; \R^M)$, from \eqref{tau-LB} we readily see that
\[
\f(u) \ge \int_{\partial \o}\tau(x, \tr u(x))\,d\hno \ge - \|c\|_{L^1(\partial \o)} - \|L\|_{L^\infty(\partial \o)}\int_{\partial \o}|\tr u(x)|\,d\hno > - \infty.
\]
However, since we do not prescribe any control on $\tau$ from above, it is worth noting that $\f$ is not necessarily finite on $W^{1,1}(\o; \R^M)$. To see this, consider for example $\tau(x, p) \coloneqq |p|^2$. Then $\f(u) = \infty$ for every $u \in W^{1,1}(\o; \R^M)$ such that $\tr u \in L^1(\partial \o; \R^M) \setminus L^2(\partial \o; \R^M)$.

As previously remarked (see \eqref{E1} and \eqref{E2}), without additional assumptions on $\partial \o$ and $\tau$, the functional $\f$ fails, in general, to be lower semicontinuous with respect to the strong topology of $L^1(\o; \R^M)$. Thus, we are lead to consider the relaxed functional $\overline{\f} \colon L^1(\o; \R^M) \to [-\infty, \infty]$, which is classically defined via
\begin{equation}
\label{Fbar}
\overline{\f}(u) \coloneqq \inf \left\{ \liminf_{n \to \infty} \f(u_n) : u_n \to u \text{ in } L^1(\o;\R^M)\right\}.
\end{equation}

The main purpose of this paper is to provide an integral representation formula for $\overline{\f}$. To this end, consider 
\begin{equation}
\label{eq:sigma}
\hat{\tau}(x, p) \coloneqq \inf \left\{\tau(x, q) + \sigma  |p - q| : q \in \R^M \right\},
\end{equation}
and observe that if $\|L\|_{L^{\infty}(\partial \o)} \le \sigma$ then \eqref{eq:sigma} defines a Carath\'eodory function on $\partial \o \times \R^M$ which also satisfies the lower bound \eqref{tau-LB}. Furthermore, notice that for $\hno$-a.e.\@ $x \in \partial \o$, the function $\hat{\tau}(x, \cdot) \colon \R^M \to \R$ coincides with the so-called $\sigma$-Yosida transform of $\tau(x, \cdot)$, and corresponds therefore to the greatest $\sigma$-Lipschitz function below $\tau(x, \cdot)$. Thus, we define
\begin{equation}
\label{H}
\h(u) \coloneqq \left\{
\begin{array}{ll}
\displaystyle  \sigma |Du|(\o) + \int_{\partial \o} \hat{\tau}(x,\tr u(x))\,d\hno & \text{ if } u \in BV(\o; \R^M), \\
&\\
\infty & \text{ otherwise in } L^1(\o; \R^M).
\end{array}
\right.
\end{equation}

Before we state our main result, we give two definitions. The first is that of open set with boundary almost of class $C^1$ (compare it with the definition in Section 9.3 of \cite{MR2976521}).

\begin{definition}
\label{a-C1}
Let $\o$ be an open subset of $\R^N$. We say that $\partial \o$ is \emph{almost of class $C^1$} if it is Lipschitz continuous and there exists a closed set $\mathcal{S} \subset \partial \o$ such that 
\[
\hno(\mathcal{S}) = 0,
\]
and with the property that for every $z \in \partial \o \setminus \mathcal{S}$ there exist $R > 0$, $i \in \{1, \dots, N\}$, and a function $f \colon \R^{N - 1} \to \R$ of class $C^1$ such that either 
\[
\o \cap B(z, R) = \{x \in B(z, R) : x_i > f(x_i')\}
\]
or 
\[
\o \cap B(z, R) = \{x \in B(z, R) : x_i < f(x_i')\},
\]
where $x'_i$ is the point of $\R^{N - 1}$ obtained by removing the $i$-th entry, namely $x_i$, from $x$.
\end{definition}

\begin{remark}
Note that the class of open sets given in \Cref{a-C1} includes all Lipschitz domains with boundary piecewise of class $C^1$. Indeed, this subclass corresponds to case where the singular set $\mathcal{S}$ satisfies $\mathcal{H}^{N - 2}(\mathcal{S}) < \infty$.
\end{remark}

Next, we introduce the object that encodes the influence of the geometry of $\o$ (see Definition 1 in \cite{MR555952}; see also \cite{MR482506}).

\begin{definition}
\label{q-def}
Let $\o$ be an open subset of $\R^N$ with Lipschitz continuous boundary. For $x \in \partial \o$, we define 
\[
q_{\partial\o}(x) \coloneqq  \lim_{\rho \to 0^+} \sup\left\{ \frac{1}{|D\ca_E|(\o)} \int_{\partial \o}\ca_E(x)\,d\hno  : E \subset B(x, \rho),\ \mathcal{L}^N(E) > 0,\ |D\ca_E|(\o) < \infty \right\},
\]
where $\ca_E$ denotes the characteristic function of the set $E$.
\end{definition}

Our main result reads as follows.

\begin{theorem}
\label{thm:main}
Let $\o$ be an open bounded subset of $\R^N$ with Lipschitz continuous boundary.
Given a nonnegative function $c \in L^1(\partial \o)$, $\sigma > 0$, and a continuous function $L \colon \partial \o \to [0,\infty)$ such that $\|L\|_{L^\infty(\partial\o)} \le \sigma$, let $\tau \colon \partial \o \times \R^M \to [- \infty, \infty)$ be a Carath\'eodory function as in \eqref{tau-LB}. Moreover, let $\f$, $\overline{\f}$, and $\h$ be given as in \eqref{F}, \eqref{Fbar}, and \eqref{H}, respectively. Assume furthermore that 
\begin{equation}
\label{F<infty}
\f (\tilde{u}) < \infty
\end{equation}
for some $\tilde{u} \in W^{1,1}(\o; \R^M)$. Then, the following statements hold: 
\begin{itemize}
\item[$(i)$] if $\partial \o$ is of class $C^2$ then $\overline{\f}(u) = \h(u)$ for all $u \in BV(\o; \R^M)$;
\item[$(ii)$] if $\partial \o$ is almost of class $C^1$ and there exists $\e_0 > 0$ such that
\begin{equation}
\label{upperbound}
L(x) q_{\partial \o}(x) \le (1 - 2\e_0) \sigma
\end{equation} 
holds for all $x\in\partial\o$, then $\overline{\f}(u) = \h(u)$ for all $u \in L^1(\o; \R^M)$.
\end{itemize}
\end{theorem}

\begin{remark}
In \Cref{sec-normal} we present an important extension of \Cref{thm:main} that allows us to consider surface densities $\tau$ which are not necessarily continuous in the second variable. To be precise, under a mild integrability condition, we prove a representation formula for the relaxation in $BV$ of the functional $\f$ when $\tau$ is a Borel function which is only upper semicontinuous as a function of the variable $p$.
\end{remark}

\subsection{Discussion of the assumptions and additional remarks}
Let us now comment on the main assumptions in \Cref{thm:main}. We begin by observing that the local behavior of $\partial \o$ enters the relaxation procedure through condition \eqref{upperbound}. This in turn can be understood as a restriction on $L$, and therefore (see \eqref{tau-LB}) on the class of surface densities for which the representation formula 
\[
\overline{\f} = \h
\]
holds. It is important to notice, however, that the geometry of the set has no effect on the regularization parameter in the Yosida transform $\hat{\tau}$ (see \eqref{eq:sigma}).

Additionally, we note that condition \eqref{upperbound} is required in order to apply the weighted trace inequality
\[
\int_{\partial \o} s(x)|\tr u(x)|\,d\hno \le (1 - \e)|Du|(\o) + C \int_{\o}|u(x)|\,dx,
\]
which holds for every $u \in BV(\o;\R^M)$ provided that, for $\e > 0$, the function $s \colon \partial \o \to [0, \infty)$ is continuous and satisfies
\begin{equation}
\label{trace-compare}
s(x)q_{\partial\o}(x) \le 1 - 2 \e.
\end{equation}
Here $C$ is a positive constant which only depends on $\o$, $\e$, and $s$. Compare, indeed, \eqref{upperbound} with \eqref{trace-compare}. This trace inequality was first obtained by Giusti (see Lemma 1.2 in \cite{MR482506}; see also \Cref{local-AG} below) and constitutes one of the key tools in our proof of the liminf inequality. It is worth noting also that for domains with boundary of class $C^2$ we use a sharper version of this inequality (see \Cref{trace-C2}) which, roughly speaking, allows to take $\e_0 = 0$ in \eqref{upperbound}. We refer to \Cref{rem-sharp} for more details. \\

Next, we observe that the lower bound \eqref{tau-LB}, together with the assumption that $\|L\|_{L^\infty(\partial\o)}\leq \sigma$, is in some sense optimal in order to have a nontrivial result. Indeed, assume that $\tau \colon \partial \o \times \R^M \to \R$ is a Carath\'eodory function such that
\[
\essinf_{x\in\partial\o} \liminf_{|p|\to\infty} \frac{\tau(x,p)}{|p|} < -\sigma.
\]
Then, as one can readily check, for each $u \in L^1(\o;\R^M)$ it is possible to find a sequence $\{u_n\}_{n \in \N}$ of functions in $W^{1,1}(\o;\R^M)$ with $u_n \to u$ in $L^1(\o; \R^M)$ and such that $\f(u_n) \to -\infty$. In particular, this implies that $\overline{\f} \equiv -\infty$.
\\

It is interesting to observe that in general there is no compactness for energy bounded sequences. Indeed, if for instance $\tau$ is bounded from above, the sequence given by functions of the form $u_n \equiv n$, for $n \in \N$, has uniformly bounded energy, but no convergent subsequence. A more critical loss of compactness in $BV$ occurs if the inequality $\|L\|_{L^{\infty}(\partial \o)} \le \sigma$ fails to be strict. To see this, fix $N = M = 1$, $\sigma = 1$, and consider $\o \coloneqq (0,1)$ and $\tau(x, p) \coloneqq p$. For $n \in \N$, let $u_n$ be defined via
\[
u_n(x) \coloneqq
\left\{
\begin{array}{ll}
\log x & \text{ for } x \in \left(\frac{1}{n}, 1 \right),\\
& \\
\log \frac{1}{n} & \text{ for } x \in \left(0, \frac{1}{n} \right).
\end{array}
\right.
\]
Then $u_n \in W^{1,1}(0,1)$ and by means of a direct computation we see that
\[
\f(u_n) = \int_0^1 |u_n'(x)|\,dx + u_n(1) + u_n(0) = 0.
\]
In particular, since $u_n \to \log$ in $L^1(0,1)$, this shows that $\overline{\f}$ can be finite even for functions in $L^1(0,1) \setminus BV(0,1)$. For this reason, in the case of domains with boundary of class $C^2$, we only provide an integral representation in $BV(\o;\R^M)$. We remark, however, that condition \eqref{upperbound} allows us to obtain a uniform bound on the gradients of energy bounded sequences. This is achieved in \Cref{lem:cmpt} and prevents the situation described above from happening. Hence, in this case we can show that the representation formula $\overline{\f} = \h$ holds in $L^1(\o;\R^M)$.
\\


Finally, it is natural to ask whether a similar analysis can be carried out for general Lipschitz domains. We plan to address this question in a forthcoming paper. A partial result in this direction, which requires a rather stringent (but classical) condition on the contact energy $\tau$, is presented in \Cref{Lip-dom}. We refer to \Cref{lip-rem} for more details.


\subsection{Plan of the paper}
The paper is organized as follows. In \Cref{Prelim-sec} we introduce the relevant notation and present a series of useful technical results that will be needed throughout the paper. 
Particularly important for our purposes are the sharp trace inequalities presented in \Cref{sect-sharp-ineq}. We mention here \Cref{thm:trace} and \Cref{local-AG}, which will be used in \Cref{sec-liminf} to prove the liminf inequality in the case of domains with boundary of class $C^2$ and almost of class $C^1$, respectively. \Cref{sec-limsup} is entirely devoted to the proof of the limsup inequality. Finally, in \Cref{sec-extensions} we prove two variants of \Cref{thm:main}. In \Cref{sec-normal} we show that, under very mild additional integrability assumptions, the techniques presented can be adapted to include surface densities $\tau$ which are not necessarily continuous in the second variable. In \Cref{sec-Lip} we discuss the possibility of extending our analysis to general Lipschitz domains. To be precise, we show that the relaxed energy $\overline{\f}$ is a lower semicontinuous extension of $\f$ in $BV(\o;\R^M)$ by assuming a rather stringent Lipschitz condition on the function $\tau$.


\section{Preliminary results}
\label{Prelim-sec}
The purpose of this section is to collect some of the definitions, tools, and technical results that will be used throughout the paper, as well as to introduce most of the relevant notation. 

\subsection{Functions of bounded variation}
\label{sec-BV}
We begin by recalling basic definitions and properties of functions of bounded variation. A detailed treatment of these topics can be found, for example, in the monographs \cite{AFP}, \cite{EG}, and \cite{G}. 

\begin{definition}
Let $u \in L^1(\o; \R^M)$. We say that $u$ is of bounded variation in $\o$ if its distributional derivative $Du$ is a finite matrix-valued Radon measure on $\Omega$, i.e.\@, if 
\[
|Du|(\o) = \sup \left\{\sum_{i = 1}^M \int_{\o}u^i(x)(\div \varphi)^i(x)\,dx : \varphi \in C_c^{\infty}(\o; \R^{M \times N}), |\varphi(x)| \le 1\right\} < \infty.
\] 
We write $BV(\o; \R^M)$ to denote the vector space of functions of bounded variation in $\o$.
\end{definition}

We recall below the definition of strict convergence in $BV(\o; \R^M)$. Used in conjunction with a smoothing argument, this notion of convergence will prove useful throughout the rest of the paper to infer the counterpart in $BV(\o; \R^M)$ to several identities for Sobolev functions. 

\begin{definition}
Let $u \in BV(\o; \R^M)$. We say that a sequence $\{u_n\}_{n \in \N} \subset BV(\o; \R^M)$ converges strictly in $BV(\o; \R^M)$ to $u$ if $u_n \to u$ in $L^1(\o; \R^M)$ and $|Du_n|(\o) \to |Du|(\o)$.
\end{definition}

\begin{theorem}
\label{GGT}
Let $\o$ be a bounded open subset of $\R^N$ with Lipschitz continuous boundary. Then there exists a continuous linear operator 
\[
\tr \colon BV(\o; \R^M) \to L^1(\partial \o; \R^M)
\]
with the following properties: 
\begin{itemize}
\item[$(i)$] $\tr u = u$ on $\partial \o$ for all $u \in BV(\o; \R^M) \cap C(\overline{\o}; \R^M)$;
\item[$(ii)$] for all $u \in BV(\o; \R^M)$ and $\varphi \in C^1(\R^N; \R^N)$
\[
\int_{\partial \o} (\varphi(x) \cdot \nu_{\partial \o}(x)) \tr u(x) \,d\hno = \int_{\o} \varphi(x) \cdot dDu(x) + \int_{\o} \div \varphi(x) u(x) \,dx,
\]
where $\nu_{\partial \o}$ denotes the outer unit normal vector to $\partial \o$.
\end{itemize}
\end{theorem}

The next result is a slight refinement of a well-known theorem by Gagliardo (see \cite{MR102739}) concerning the surjectivity of the trace operator. 

\begin{theorem}
\label{giusti-C1}
Let $\o$ be a bounded open subset of $\R^N$ with Lipschitz continuous boundary. Then, for every $\e > 0$ there exists a constant $C > 0$ such that for every $g \in L^1(\partial \o; \R^M)$ we can find a function $w \in W^{1,1}(\o;\R^M)$ having trace $g$ on $\partial \o$ and such that
\begin{align*}
\int_{\o} |w(x)|\,dx & \le \e \int_{\partial \o}|g(x)|\,d\hno, \\
\int_{\o} |\nabla w(x)|\,dx & \le C \int_{\partial \o}|g(x)|\,d\hno.
\end{align*}
Moreover, if in addition $\partial \o$ is either of class $C^1$, or almost of class $C^1$ (in the sense of \Cref{a-C1}), we can take $C = 1 + \e$. 
\end{theorem}

As also previously explained in \cite{G} (see Theorem 2.16 and Remark 2.17), \Cref{giusti-C1} can be obtained by first constructing an explicit extension of the boundary value $g$ in the special case where $\partial \o$ is a hyperplane, and conclude in the general case by standard localization and flattening arguments based on a suitably defined partition of unity. While the assertions of \Cref{giusti-C1} are well known to all specialists in the field, to the best of our knowledge a proof of this extension theorem in the present setting is not available in the literature. Since this result is pivotal to our construction of a recovery sequence, such proof is included here for completeness. Following the strategy outlined above, we will make use of the following result, for a proof of which we refer the reader to Proposition 2.15 in \cite{G}.
\begin{proposition}
\label{giusti-2.15}
Let $B_{N - 1}(0', R)$ denote the ball of radius $R$ centered at the origin of $\R^{N - 1}$ and let $g$ be a function in $L^1(B_{N - 1}(0',R); \R^M)$ with compact support. For every $\e > 0$ there exists a function $w \in W^{1,1}(B_{N - 1}(0',R) \times (0, R); \R^M)$ with trace $g$ on $B_{N - 1}(0',R)$ and such that
\begin{align*}
\int_{B_{N - 1}(0',R) \times (0, R)} |w(x)|\,dx & \le \e \int_{B_{N - 1}(0',R)}|g(x)|\,d\hno, \\
\int_{B_{N - 1}(0',R) \times (0, R)} |\nabla w(x)|\,dx & \le (1 + \e) \int_{B_{N - 1}(0',R)}|g(x)|\,d\hno.
\end{align*}
\end{proposition}

Before we proceed with the proof of \Cref{giusti-C1}, let us mention that here and in the following, given a set $E \subset \R^N$ and a finite collection of open sets $\{U_i\}_{i \in I}$ such that $E \subset \bigcup_{i \in I} U_i$, we say that the family of functions $\{\psi_i\}_{i \in I}$ is a partition of unity on $E$ subordinated to the covering $\{U_i\}_{i \in I}$ if $\psi_i \in C^1_c(U_i; [0, 1])$ for each $i \in I$ and furthermore
\[
\sum_{i \in I}\psi_i(x) = 1
\]
for each $x \in E$.  

\begin{proof}[Proof of \Cref{giusti-C1}]
We divide the proof into several steps.
\newline
\textbf{Step 1:} Assume first that $\partial \o$ is of class $C^1$. Then, for every $z \in \partial \o$ we can find an open set $U_z$, $z \in U_z$, a rigid motion $\varphi_z: \R^N\to\R^N$, and a function $f_z \colon \R^{N - 1} \to \R$ of class $C^1$ with 
\begin{equation}
\label{grad=0}
f_z(0') = 0, \qquad \nabla f_z(0') = 0'
\end{equation}
such that
\[
\varphi_z(\o \cap U_z) = \{y = (y', t) \in \varphi_z(U_z) : y' \in U'_{z} \text{ and } t > f_z(y')\},
\]
where $U'_{z} \subset \R^{N - 1}$ is an open neighborhood of $0'$. Moreover, let $\Psi_z \colon \R^N \to \R^N$ be defined via
\[
\Psi_z(y', t) \coloneqq (y', t - f_z(y')),
\]
and set
\[
\Phi_z(x) \coloneqq \Psi_z(\varphi_z(x)).
\]
Notice that $\Phi_z \colon \R^N \to \R^N$ is Lipschitz continuous, one-to-one, and that $|\det \nabla \Phi_z(x)| = 1$ for every $x \in \R^N$. Let $\rho > 0$ be given and let $R_z > 0$ be such that
\begin{equation}
\label{rho-A}
\max \left\{|\nabla f_{z}(y')| : y' \in \overline{B_{N - 1}(0', R_z)}\right\} \le \rho, 
\end{equation}
and 
\[
\Phi_z^{-1}\left(B_{N - 1}(0', R_z) \times (-R_z, R_z)\right) \subset U_z.
\]
Since $\partial \o$ is compact, we can find $z^1, \dots, z^k \in \partial \o$ with the property that 
\[
\partial \o \subset \bigcup_{i = 1}^k V_i, 
\]
where
\[
V_i \coloneqq \Phi_{z^i}^{-1}\left(B_{N - 1}(0', R_{z^i}) \times (-R_{z^i}, R_{z^i})\right).
\]
To keep the notation as simple as possible, throughout the rest of the proof we set $\Phi_i \coloneqq \Phi_{z^i}$, $f_i \coloneqq f_{z^i}$, and use a similar convention whenever applicable. Let $\{\psi_i\}_{i \le k}$ be a partition of unity on $\partial \o$ subordinated to the covering $\{V_i\}_{i \le k}$. Furthermore, let $g_i \coloneqq \psi_i g$ and set
\begin{equation}
\label{vi-def}
v_i(y') \coloneqq g_i(\Phi_i^{-1}(y', 0)).
\end{equation}
Then $v_i \in L^1(B_{N - 1}(0', R_i); \R^M)$ and has compact support. Thus, we are in a position to apply \Cref{giusti-2.15} for $\eta > 0$, chosen in such a way that
\begin{equation}
\label{eta-small}
\max \left\{ \|\nabla \psi_i\|_{L^{\infty}(V_i; \R^N)} : i \le k \right\} \eta^{1/2} \le 1.
\end{equation}
To be precise, for every $i \in \{1, \dots, k\}$, we can find $w_i \in W^{1,1}(B_{N - 1}(0', R_i) \times (0, R_i); \R^M)$ such that
\begin{align}
\int_{B_{N - 1}(0',R_i) \times (0, R_i)} |w_i(q)|\,dq & \le \eta \int_{B_{N - 1}(0',R_i)}|v_i(y')|\,d\hno, \label{poU-eta}\\
\int_{B_{N - 1}(0',R_i) \times (0, R_i)} |\nabla w_i(q)|\,dq & \le (1 + \eta) \int_{B_{N - 1}(0',R_i)}|v_i(y')|\,d\hno. \label{poU-1+eta}
\end{align}
Finally, set
\[
w(x) \coloneqq \sum_{i = 1}^k \psi_i(x)w_i(\Phi_i(x)).
\]
The remainder of this step is dedicated to proving that $w$ has all the desired properties. Indeed, as one can readily check, $\tr w = g$ on $\partial \o$. Moreover, the change of variables $\Phi_i(x) = q$, together with \eqref{vi-def} and \eqref{poU-eta}, yields
\begin{align}
\int_{\o} |w(x)| \,dx  
& \le \sum_{i = 1}^k \int_{B_{N - 1}(0', R_i) \times (0, R_i)} |w_i(q)|\,dq \notag \\
& \le \sum_{i = 1}^k \eta \int_{B_{N - 1}(0', R_i)}|v_i(y')|\,d\hno \notag \\
& \le \sum_{i = 1}^k \eta \int_{B_{N - 1}(0', R_i)}|v_i(y')|\sqrt{1 + |\nabla f_i(y')|^2}\,d\hno \notag \\
& = \eta  \int_{\partial \o}|g(x)|\,d\hno. \label{L1bound-e}
\end{align}
For $x \in \o \cap V_i$, let $W_i(x) \coloneqq w_i(\Phi_i(x))$. We claim that for $\mathcal{L}^N$-a.e.\@ $x \in \o \cap V_i$ we have
\begin{equation}
\label{chain-grad}
|\nabla W_i(x)| \le C_{\rho} |\nabla w_i(\Phi_i(x))|,
\end{equation}
where 
\begin{equation}
\label{C-rho}
C_{\rho} \coloneqq \sqrt{1 + (N - 1)\rho + (N - 1)\rho^2}.
\end{equation}
Since $\Phi_i \coloneqq \Psi_i \circ \varphi_i$, where $\varphi_i$ is a fixed rigid motion, it is enough to show that 
\[
|\nabla \tilde{w}_i(y)| \le C_{\rho} |\nabla w_i(\Psi_i(y))|
\]
for $\mathcal{L}^N$-a.e.\@ $y \in \varphi_i(\o \cap V_i)$, where $\tilde{w}_i(y) \coloneqq w_i(\Psi_i(y))$. To see this, we begin by observing that for all $j \in \{1, \dots, N\}$ and for $\mathcal{L}^N$-a.e.\@ $y \in \varphi_i(\o \cap V_i)$
\[
\frac{\partial \tilde{w}_i}{\partial y_j}(y) = \sum_{n = 1}^N \frac{\partial w_i}{\partial q_n}(\Psi(y))\frac{\partial \Psi^n}{\partial y_j}(y).
\]
Moreover, as one can readily check, the expression above can be rewritten as 
\[
\frac{\partial \tilde{w}_i}{\partial y_j}(y) = 
\left\{
\arraycolsep=1.4pt\def\arraystretch{2}
\begin{array}{ll}
\displaystyle \frac{\partial w_i}{\partial q_j}(\Psi(y)) - \frac{\partial w_i}{\partial q_N}(\Psi(x))\frac{\partial f_i}{\partial y_j}(y') & \text{ if } j \le N - 1, \\
\displaystyle \frac{\partial w_i}{\partial q_N}(\Psi(y)) & \text{ if } j = N.
\end{array}
\right.
\]
Thus, by Young's inequality and \eqref{rho-A}, for all $j \in \{1, \dots, N - 1\}$ and $\mathcal{L}^N$-a.e.\@ $y \in \varphi_i(\o \cap V_i)$ we have that 
\begin{align*}
\left|\frac{\partial \tilde{w}_i}{\partial y_j}(y)\right|^2 & \le \left|\frac{\partial w_i}{\partial q_j}(\Psi(y))\right|^2 + \left|\frac{\partial f_i}{\partial y_j}(y')\right|^2 \left|\frac{\partial w_i}{\partial q_N}(\Psi(y))\right|^2 + 2 \left|\frac{\partial w_i}{\partial q_j}(\Psi(y)) \frac{\partial f_i}{\partial y_j}(y') \frac{\partial w_i}{\partial q_N}(\Psi(y))\right| \\
& \le (1 + \rho)\left|\frac{\partial w_i}{\partial q_j}(\Psi(y))\right|^2 + (\rho + \rho^2)\left|\frac{\partial w_i}{\partial q_N}(\Psi(y))\right|^2,
\end{align*}
and therefore, we obtain 
\begin{align*}
|\nabla \tilde{w}_i(y)|^2 
& = \sum_{i = 1}^N \sum_{j = 1}^{N - 1}\left|\frac{\partial \tilde{w}_i}{\partial y_j}(y)\right|^2 + \sum_{i = 1}^N\left|\frac{\partial \tilde{w}_i}{\partial y_N}(y)\right|^2 \\
& \le \sum_{i = 1}^N \sum_{j = 1}^{N - 1} \left[ (1 + \rho)\left|\frac{\partial w_i}{\partial q_j}(\Psi(y))\right|^2 + (\rho + \rho^2) \left|\frac{\partial w_i}{\partial q_N}(\Psi(y))\right|^2 \right] + \sum_{i = 1}^N \left|\frac{\partial w_i}{\partial q_N}(\Psi(y))\right|^2 \\
& = \sum_{i = 1}^N \sum_{j = 1}^{N - 1}(1 + \rho)\left|\frac{\partial w_i}{\partial q_j}(\Psi(y))\right|^2  + \sum_{i = 1}^N (1 + (N - 1)\rho + (N - 1)\rho^2) \left|\frac{\partial w_i}{\partial q_N}(\Psi(y))\right|^2 \\
& \le (1 + (N - 1)\rho + (N - 1)\rho^2)|\nabla w_i(\Psi(y))|^2.
\end{align*}
This concludes the proof of \eqref{chain-grad}. Since
\begin{equation}
\label{grad-PoU}
\int_{\o} |\nabla w(x)|\,dx \le \sum_{i = 1}^k \int_{\o \cap V_i}|\nabla \psi_i(x)||w_i(\Phi(x))|\,dx + \sum_{i = 1}^k \int_{\o \cap V_i} \psi_i(x)|\nabla W_i(x)|\,dx,
\end{equation}
reasoning as in \eqref{L1bound-e} and thanks to our choice of $\eta$ (see \eqref{eta-small}), we deduce that  
\begin{align}
\label{grad-PoU-1}
\sum_{i = 1}^k \int_{\o \cap V_i}|\nabla \psi_i(x)||w_i(\Phi(x))|\,dx & \le \max_i\|\nabla \psi_i(x) \|_{L^{\infty}(V_i; \R^N)}\sum_{i = 1}^k \int_{\o \cap V_i}|w_i(\Phi(x))|\,dx \notag \\
& \le \eta \max_i\|\nabla \psi_i(x) \|_{L^{\infty}(V_i; \R^N)}  \int_{\partial \o}|g(x)|\,d\hno \notag \\
& \le \eta^{1/2} \int_{\partial \o}|g(x)|\,d\hno.
\end{align}
Moreover, it follows from \eqref{chain-grad} that
\begin{align}
\label{grad-PoU-2}
\sum_{i = 1}^k \int_{\o \cap V_i} \psi_i(x)|\nabla W_i(x)|\,dx & \le C_{\rho} \sum_{i = 1}^k \int_{\o \cap V_i}|\nabla w_i(\Phi(x))|\,dx \notag \\
& \le C_{\rho} \sum_{i = 1}^k \int_{B_{N - 1}(0', R_i) \times (0, R_i)}|\nabla w_i(q)|\,dq \notag \\
& \le C_{\rho} (1 + \eta) \sum_{i = 1}^k \int_{B_{N - 1}(0', R_i)}|v_i(y')|\,d\hno \notag \\
& \le C_{\rho} (1 + \eta) \int_{\partial \o}|g(x)|\,d\hno, 
\end{align}
where in the second to last inequality we have used \eqref{poU-1+eta}. Combining \eqref{grad-PoU}, \eqref{grad-PoU-1}, and \eqref{grad-PoU-2} yields 
\begin{equation}
\label{gradbound-e}
\int_{\o} |\nabla w(x)|\,dx \le (\eta^{1/2} + C_{\rho}(1 + \eta)) \int_{\partial \o}|g(x)|\,d\hno.
\end{equation}
Finally, given $\e > 0$, from \eqref{L1bound-e} and \eqref{gradbound-e} we deduce that to conclude it is enough to choose $\rho$ and $\eta$ in such a way that $\eta \le \e$ and $\eta^{1/2} + C_{\rho}(1 + \eta) \le 1 + \e$ (see \eqref{C-rho} for the definition of $C_{\rho}$).
\newline
\textbf{Step 2:} If $\partial \o$ is Lipschitz continuous, the proof requires only minimal changes. Indeed, observe that since $\o$ is bounded, there exists a constant $L_{\partial \o} > 0$ such that for every $z \in \partial \o$ we can find a set of local coordinates and a Lipschitz continuous function $f_z$ with $\Lip f_z \le L_{\partial \o}$ such that $\partial \o$ coincides with the graph of $f_z$ in the new coordinate system. To be precise, it suffices to proceed as in the previous step, with the exception that we do not require that $\nabla f_z(0') = 0'$ in \eqref{grad=0}, and by selecting a rigid motion $\varphi_z$ in such a way that 
\[
\|\nabla f_z\|_{L^{\infty}(B_{N - 1}(0', R_z); \R^{N - 1})} \le L_{\partial \o}.
\]
This estimate is then used in place of \eqref{rho-A}. 
\newline
\textbf{Step 3:} Suppose now that $\partial \o$ is almost of class $C^1$ and let $\mathcal{S}$ be as in \Cref{a-C1}. Given $\e > 0$, let $\gamma$ be a positive constant, which we choose later. Reasoning as in the proof of Theorem 9.6 in \cite{MR2976521}, we can find a countable subset of $\mathcal{S}$, namely $\{z_n\}_{n \in \N} \subset \mathcal{S}$, such that 
\[
\mathcal{S} \subset \mathcal{B}_{\gamma} \coloneqq \bigcup_{n \in \N} B(z_n, \gamma_n),
\]
where each of the $\gamma_n$ is chosen in such a way that (up to a rotation) $\partial \o$ coincides with the graph of a Lipschitz function $f_n$ in $B(z_n, \gamma_n)$ with $\Lip f_n \le L_{\partial \o}$, and furthermore 
\begin{equation}
\label{gamma-n}
\sum_{n \in \N} \gamma_n^{N - 1} < \gamma.
\end{equation}
On the other hand, for every $z \in \partial \o \setminus \mathcal{S}$ we can find $\Phi_z$, $f_z$, and $R_z$ as in Step 1, in such a way that \eqref{rho-A} holds with $\rho = \e/2$. Notice that 
\[
\partial \o \subset \mathcal{B}_{\gamma} \cup \bigcup_{z \in \partial \o \setminus \mathcal{S}} \Phi_z^{-1}(B_{N - 1}(0', R_z) \times (- R_z, R_z)).
\]
Thus, extracting a finite subcover of $\partial \o$ and arguing as in the previous steps, we can find a function $w_{\gamma}$ with trace $g$ on $\partial \o$ such that 
\[
\int_{\o}|w_{\gamma}(x)|\,dx \le \e \int_{\partial \o}|g(x)|\,d\hno
\]
and
\[
\int_{\o}|\nabla w_{\gamma}(x)|\,dx \le \left(1 + \frac{\e}{2}\right)\int_{\partial \o}|g(x)|\,d\hno + C \int_{\partial \o \cap \mathcal{B}_{\gamma}}|g(x)|\,d\hno.
\]
Notice that the constant $C$, given as in the previous step, depends only on $\partial \o$ through $L_{\partial \o}$, and in particular is independent of $\gamma$. To conclude, it is enough to notice that since $\hno(\partial \o \cap \mathcal{B}_{\gamma}) \to 0$ as $\gamma \to 0$ (see \eqref{gamma-n}), 
\[
C \int_{\partial \o \cap \mathcal{B}_{\gamma}}|g(x)|\,d\hno \le \frac{\e}{2} \int_{\partial \o}|g(x)|\,d\hno
\]
for all $\gamma$ sufficiently small. This concludes the proof.
\end{proof}

Next, we recall two approximation results.

\begin{lemma}
\label{smoothapprox}
Let $\o$ be a bounded open subset of $\R^N$ with Lipschitz continuous boundary. Then, for every $u \in BV(\o; \R^M)$ there exists a sequence $\{u_n\}_{n \in \N}$ of functions in $W^{1,1}(\o; \R^M) \cap C(\overline{\o}; \R^M)$ such that $\tr u_n \to \tr u$ in $L^1(\partial \o; \R^M)$ and $u_n \to u$ strictly in $BV(\o;\R^M)$.
\end{lemma}

\begin{lemma}
\label{11approx}
Let $\o$ be a bounded open subset of $\R^N$ with Lipschitz continuous boundary. Then, for every $u \in BV(\o; \R^M)$ there exists a sequence $\{u_n\}_{n \in \N}$ of functions in $W^{1,1}(\o; \R^M)$ such that $\tr u_n = \tr u$ and $u_n \to u$ strictly in $BV(\o;\R^M)$.
\end{lemma}

For a proof of \Cref{smoothapprox} we refer to Remark 3.22 and Theorem 3.88 in \cite{AFP}; \Cref{11approx} corresponds to Lemma 2.5 in \cite{BFMglobal}.

The following result is due to Bouchitt\'e, Fonseca, and Mascarenhas (see Lemma 2.1 in \cite{MR1885119}), and is of key importance in our construction of a recovery sequence. We report here the proof for the reader's convenience. 

\begin{lemma}
\label{IF-lem}
Let $\o$ be a bounded open subset of $\R^N$ and assume that $\partial \o$ is almost of class $C^1$. Then, for every $u \in BV(\o; \R^M)$ and every $p \in L^1(\partial \o; \R^M)$, there exists a sequence $\{u_n\}_{n \in \N}$ of functions in $W^{1,1}(\o; \R^M)$ such that $\tr u_n = p$, $u_n \to w$ in $L^1(\o; \R^M)$, and 
\[
\limsup_{n \to \infty} \int_{\o}|\nabla u_n(x)|\,dx \le |Du|(\o) + \int_{\partial \o}|p(x) - \tr u(x)|\,d\hno.
\] 
\end{lemma}
\begin{proof}
Set $g \coloneqq p - \tr u$. Then, by \Cref{giusti-C1}, for every $n \in \N$ we can find a function $w_n \in W^{1,1}(\o; \R^M)$ such that $\tr w_n = g$ and furthermore 
\begin{align*}
\int_{\o} |w_n(x)|\,dx & \le \frac{1}{n} \int_{\partial \o} |g(x)|\,d\hno, \\
\int_{\o} |\nabla w_n(x)|\,dx & \le \left(1 + \frac{1}{n}\right) \int_{\partial \o} |g(x)|\,d\hno.
\end{align*}
Moreover, \Cref{11approx} gives the existence of a sequence $\{v_n\}_{n \in \N} \subset W^{1,1}(\o; \R^M)$ with $\tr v_n = \tr u$ and such that $v_n \to u$ strictly in $BV(\o; \R^M)$. Set $u_n \coloneqq v_n + w_n$. Then $u_n \in W^{1,1}(\o; \R^M)$, $\tr u_n = p$, and furthermore 
\begin{align*}
\limsup_{n \to \infty} \int_{\o} |\nabla u_n(x)|\,dx & \le \limsup_{n \to \infty}\int_{\o}|\nabla v_n(x)|\,dx + \limsup_{n \to \infty} \int_{\o} |\nabla w_n(x)|\,dx \\
& \le |Du|(\o) + \int_{\o}|p(x) - \tr u(x)|\,d\hno.
\end{align*}
Similarly, one can show that $u_n \to u$ in $L^1(\o; \R^M)$; thus, the sequence $\{u_n\}_{n \in \N}$ has all the desired properties. 
\end{proof}


\subsection{Sharp trace inequalities}
\label{sect-sharp-ineq}
In this subsection we record some fundamental results concerning the attainability of a trace inequality of the form
\begin{equation}
\label{QC}
\int_{\partial \o} |\tr u(x)|\,d\hno \le Q |Du|(\o) + C\int_{\o}|u(x)|\,dx,
\end{equation}
where $Q$ and $C$ are positive constants, independent of $u$.
Furthermore we discuss the existence of an optimal constant $Q$, denoted below by $Q_{\partial \o}$, and report its explicit value for certain geometries. A refined version of inequality \eqref{QC}, where the local geometry of the set $\o$ is taken into consideration, is also presented below (see \Cref{local-AG}). 

In the classical paper \cite{MR555952}, Anzellotti and Giaquinta identified necessary and sufficient conditions that characterize the class of sets for which the trace operator is well defined and continuous with respect to strict convergence in $BV$. Roughly speaking, the sets which satisfy these conditions are the ones for which an inequality of the form \eqref{QC} holds. Before we state their precise results, let us remark that the definition of $q_{\partial \o}$ (see \Cref{q-def}) can be readily extended to include the case where $\o$ is a Caccioppoli set. Moreover, throughout the following we let 
\begin{equation}
\label{Q-def}
Q_{\partial\o} \coloneqq \sup \left\{ q_{\partial\o}(x) : x \in \partial \o \right\}.
\end{equation}
The following theorem gives a precise connection between the constant $Q_{\partial \o}$ and the trace inequality \eqref{QC}. For a proof we refer to Theorem 4 and Theorem 5 in \cite{MR555952}.

\begin{theorem}
\label{thm:trace}
Let $\o$ be an open subset of $\R^N$ with $\hno(\partial \o) < \infty$ and $Q_{\partial\o} < \infty$. Then, for every $\e > 0$ there exists $C(\o,\e) > 0$ such that
\begin{equation}\label{eq:traceinequality}
\int_{\f \o} |\tr u(x)| \,d\hno \le (Q_{\partial \o} + \e) |Du|(\o) + C(\o, \e) \int_\o |u(x)| \,dx
\end{equation}
for all $u \in BV(\o;\R^M)$. Here $\f \o$ denotes the reduced boundary of $\o$. Moreover, if in addition $\o$ is bounded and satisfies $|D\ca_{\o}|(\R^N) = \hno(\partial \o)$, then an inequality of the form \eqref{QC} holds true if and only if $Q_{\partial\o} < \infty$. In particular, $Q_{\partial \o}$ is the infimum among all the constants $Q$ for which there exists $C$ with the property that a bound of the form \eqref{QC} holds for every $u \in BV(\o; \R^M)$.
\end{theorem}

\begin{remark}\label{rem:Q}
Note that, roughly speaking, the value of $q_{\partial\o}(x)$ quantifies the best way to locally encapsulate $\partial\o$ with a set $E$ in a neighborhood of $x$. The function $q_{\partial\o}$, and consequently the constant $Q_{\partial \o}$, can be computed explicitly (or at least estimated) in several cases of interest. We record some of these computations below.
\begin{itemize}
\item[$(i)$] If $\partial \o$ is of class $C^1$ in a neighborhood of $x$ then $q_{\partial\o}(x) = 1$ (see Proposition 1.4 in \cite{MR482506}).
\item[$(ii)$] Let $N = 2$ and assume that $\o \coloneqq \{(x_1, x_2) : x_2 > \ell |x_1|\}$. Then (see at the end of Section 1 in \cite{MR482506})
\[
q_{\partial\o}(0) = 
\left\{
\begin{array}{ll}
\displaystyle \sqrt{1 + \ell^2} & \text{ if } \ell > 0, \\
\displaystyle 1 & \text{ if } \ell \le 0. 
\end{array}
\right.
\]
In particular, it is interesting to notice that $Q_{\partial \o}$ can be strictly smaller than the Lipschitz constant of the domain (see \Cref{lipQ}).
\item[$(iii)$] If $\o$ is an open subset of $\R^N$ with uniformly Lipschitz continuous boundary, then we have $Q_{\partial \o} < \infty$ (see, for example, Theorem 18.22 in \cite{L}).
\end{itemize}
\end{remark}

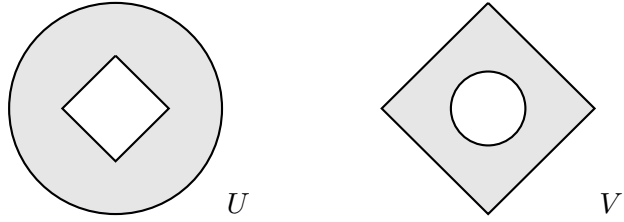
\begin{figure}
\begin{center}
\begin{tikzpicture}[blend group=normal, scale=0.7]
\draw [thick, fill = black, fill opacity = 0.1] (3, 3) circle (2cm);
\draw [thick, fill = white] (3,4) -- (4, 3) -- (3, 2) -- (2, 3) -- (3, 4);
\node [right] at (4.9, 1.2) {$U$};
\draw [thick, fill = black, fill opacity = 0.1] (10,5) -- (12, 3) -- (10, 1) -- (8, 3) -- (10, 5);
\draw [thick, fill = white] (10, 3) circle (0.7cm);
\node [right] at (11.9, 1.2) {$V$};
\end{tikzpicture}
\caption{An example of two domains with the same Lipschitz constant, but different trace embedding constant. Indeed, $Q_{\partial U} = 1$ and $Q_{\partial V} = \sqrt{2}$.}
\label{lipQ}
\end{center}
\end{figure}

The following lemma is due to Giusti (see Lemma 1.2 in \cite{MR482506}).
\begin{lemma}
\label{local-AG}
Let $\o$ be a bounded open subset of $\R^N$ with Lipschitz continuous boundary. For every $\e > 0$ and every continuous function $s \colon \partial \o \to [0, \infty)$ such that 
\begin{equation}
\label{sq}
s(x)q_{\partial\o}(x) \le 1 - 2 \e
\end{equation}
for all $x\in\partial\o$, there exists a constant $C(\o, \e, s) > 0$ with the property that for all $u \in BV(\o; \R^M)$ we have
\[
\int_{\partial \o} s(x)|\tr u(x)|\,d\hno \le (1 - \e)|Du|(\o) + C(\o, \e, s) \int_{\o}|u(x)|\,dx.
\]
\end{lemma}

\begin{proof}
Fix $\delta > 0$ such that 
\begin{equation}
\label{delta-small-def}
\delta (\|s\|_{L^{\infty}(\partial \o)} + Q_{\partial \o}) + \delta^2 \le \e.
\end{equation}
By the definition of $q_{\partial\o}$ (see \Cref{q-def}), for each $z \in \partial \o$ we can find $R_z > 0$ be such that for every $E \subset B(z, R_z)$  
\begin{equation}
\label{qr}
\int_{\partial \o} \tr \ca_E(x) \,d\hno \le (q_{\partial\o}(z) + \delta) |D\ca_{E}|(\o).
\end{equation}
Moreover, since by assumption $s$ is continuous, by eventually replacing $R_z$ with a smaller number we can assume that $s(x) \le s(z) + \delta$ for all $x \in B(z, R_z)$.
Assume first that $u \in W^{1,1}(\o; \R^M) \cap C(\overline{\o}; \R^M)$ and that its support is contained in the ball $B(z, R_z)$. Then, by the layer cake formula, \eqref{qr}, and the coarea formula we get
\begin{align}
\int_{\partial \o} |u(x)|\,d\hno & = \int_0^{\infty} \hno (\{x \in \partial \o : |u(x)| > t\})\,dt \notag \\
& = \int_0^{\infty} \int_{\partial \o} \ca_{\{|\tr u| > t\}}(x)\,d\hno dt \notag \\
& \le (q_{\partial \o}(z) + \delta) \int_0^{\infty}  |D\ca_{\{|u| > t\}}|(\o)\,dt \notag \\
& = (q_{\partial \o}(z) + \delta) \int_0^{\infty} \hno (\{x \in \o : |u(x)| = t\})\,dt \notag \\
& = (q_{\partial \o}(z) + \delta) \int_{\o}|\nabla u(x)|\,dx. \label{giusti-proof}
\end{align}
Therefore, combining \eqref{sq}, \eqref{delta-small-def}, and \eqref{giusti-proof} we arrive at
\begin{equation}
\label{supp-in-B}
\int_{\partial \o} s(x)|u(x)|\,d\hno \le (s(z) + \delta)(q_{\partial\o}(z) + \delta) \int_{\o}|\nabla u(x)|\,dx \le (1 - \e)\int_{\o}|\nabla u(x)|\,dx.
\end{equation}
To remove the assumption that the support of $u$ is contained in $B(z, R_z)$, we consider the collection of sets $\{B(z, R_z)\}_{z \in \partial \o}$ and extract a finite cover of $\partial \o$, namely $\{B_i\}_{i \le k}$. Let $\{\psi_i\}_{i \le k}$ be a partition of unity on $\partial \o$ subordinated to $\{B_i\}_{i \le k}$. We can then apply \eqref{supp-in-B} to each $u_i \coloneqq u \psi_i$ and conclude as in the proof of \Cref{giusti-C1}. The estimate for a general function $u \in BV(\o; \R^M)$ is a direct consequence of \Cref{smoothapprox}.
\end{proof}

In view of \Cref{rem:Q} $(i)$, if $\partial \o$ is of class $C^1$ we may then take $Q_{\partial \o} = 1$. This result appears also in a more recent paper by Motron (see Theorem 2.7 in \cite{MR1893737}), where it is derived from the following local version of \eqref{eq:traceinequality}. As a note to the reader, we believe it is important to mention that the cited result is stated for a domain $\o$ that is bounded and \emph{piecewise} of class $C^1$. However, this is not true in general as shown by (a suitable modification of) Modica's example in the square (see Remark 1.3 in \cite{MR921549}; see also \eqref{E1}). 

\begin{proposition}
\label{local-trace}
Suppose $U \coloneqq \{x = (x', x_N) : x_N > f(x')\}$, where $f \colon \R^{N - 1} \to \R$ is of class $C^1$. Then, for every $u \in BV(U; \R^M)$ we have 
\[
\int_{\partial U}|\tr u(x)|\,d\hno \le \sqrt{1 + \|\nabla f\|_{L^{\infty}}} |Du|(U).
\]
\end{proposition}


The next lemma can be regarded as a Gauss--Green formula for vector-valued $BV$ functions in a regular domain with boundary of class $C^2$ (compare it with \Cref{GGT}). The result is due to Fonseca and Leoni (see Lemma 2.1 in \cite{MR1656999}).

\begin{lemma}
\label{FL-lem}
Let $\o$ be a bounded open subset of $\R^N$ and assume that $\partial \o$ is of class $C^2$. Then there exists $\varphi \in C^1_0(\R^N; \R^N)$ with $|\varphi(x)| < 1$ in $\o$ such that for any $u \in BV(\o; \R^M)$
\[
\int_{\partial \o}\tr u(x)\,d\hno = \int_{\o} \varphi(x) \cdot dDu(x) + \int_{\o} \div \varphi(x) u(x)\,dx.
\]
\end{lemma}

The following is a direct consequence of \Cref{FL-lem}.

\begin{theorem}
\label{trace-C2}
Let $\o$ be a bounded open subset of $\R^N$ and assume that $\partial \o$ is of class $C^2$. Then there exists a constant $C > 0$, which only depends on $\o$, such that
\[
\int_{\partial \o}|\tr u(x)|\,d\hno \le |Du|(\o) + C(\o) \int_{\o}|u(x)|\,dx
\]
holds for every $u \in BV(\o; \R^M)$.
\end{theorem}

\begin{proof}
Let $u \in W^{1,1}(\o; \R^M) \cap C(\overline{\o}; \R^M)$ and set $v(x) \coloneqq |u(x)|$. Then $v \in W^{1,1}(\o) \cap C(\overline{\o})$ and by \Cref{FL-lem} (applied to the scalar-valued function $v$), we obtain
\begin{align*}
\int_{\partial \o}|u(x)|\,d\hno & = \int_{\partial \o} v(x)\,d\hno \\
& = \int_{\o} \varphi(x) \cdot \nabla v(x)\,dx + \int_{\o}\div \varphi(x) v(x)\,dx \\
& \le \int_{\o}|\nabla v(x)|\,dx + \|\div \varphi\|_{L^\infty(\R^N)} \int_{\o}|u(x)|\,dx \\
& \le \int_{\o}|\nabla u(x)|\,dx + \|\div \varphi\|_{L^\infty(\R^N)} \int_{\o}|u(x)|\,dx.
\end{align*}
This shows that the desired inequality holds for every $u \in W^{1,1}(\o; \R^M) \cap C(\overline{\o}; \R^M)$. The estimate in the general case follows from an application of \Cref{smoothapprox}.
\end{proof}

\begin{remark}
\label{rem-sharp}
The proof of the liminf inequality, as presented in the next section, relies heavily on the trace inequalities provided by \Cref{local-AG} and \Cref{local-trace} for the case of almost $C^1$ domains (see \Cref{liminf-C1pw}), and \Cref{trace-C2} for domains with boundary of class $C^2$ (see \Cref{liminf-C2}). 

It is worth noting that while $Q_{\partial \o} = 1$ even when $\partial \o$ is only of class $C^1$, \Cref{trace-C2} shows that, with some additional regularity on $\partial \o$, it is possible to choose the constant $C(\o, \e)$ in \Cref{thm:trace} in such a way that 
\[
\limsup_{\e \to 0^+}C(\o, \e) < \infty.
\] 
This discrepancy is reflected in the assumptions of \Cref{thm:main}. Indeed, for domains with boundary of class $C^2$ we only require that $\|L\|_{L^\infty(\partial \o)} \le \sigma$, while for domains of class $C^1$ (or $C^{1, \alpha}$ with $\alpha \in (0, 1]$, but not $C^2$) condition \eqref{upperbound} is equivalent to $\|L\|_{L^{\infty}(\partial \o)} < \sigma$.
\end{remark}


\subsection{Carath\'eodory integrands and Scorza Dragoni's property}
The purpose of this subsection is to specify the class of surface energy densities $\tau$ that we consider in the core sections of this paper. We begin by recalling the notion of Carath\'eodory integrand.

\begin{definition}
Let $\o$ be a bounded open subset of $\R^N$ with Lipschitz continuous boundary and let $\tau \colon \partial \o \times \R^M \to [-\infty, \infty)$. We say that $\tau$ is Carath\'eodory if: 
\begin{itemize}
\item[$(i)$] for $\hno$-a.e.\@ $x \in \partial \o$, the function $\tau(x, \cdot)$ is continuous;
\item[$(ii)$] for every $p \in \R^M$, the function $\tau(\cdot, p)$ is measurable.
\end{itemize}
\end{definition}

Notice that since $\partial \o$ coincides locally with the graph of a Lipschitz function, this notion can be reconducted to that given, for example, in \cite{FL}. In addition to several of the tools that were introduced above, our proof of the limsup inequality will also make use of the following version of Scorza Dragoni's theorem. 

\begin{theorem}
\label{scorza-dragoni}
Let $\o$ be a bounded open subset of $\R^N$ with Lipschitz continuous boundary and let $\tau \colon \partial \o \times \R^M \to [-\infty, \infty)$ be a Carath\'eodory function. Then, for every $\e > 0$ there exists a compact subset of $\partial \o$, namely $C_{\e}$, such that $\hno(\partial \o \setminus C_{\e}) < \e$ and with the property that the restriction of $\tau$ to $C_{\e} \times \R^M$ is continuous. 
\end{theorem}



\section{Liminf inequality}
\label{sec-liminf}
This section is dedicated to the proof of the liminf inequality. We begin by addressing the case of a regular domain with boundary of class $C^2$. Our proof is reminiscent of the work of Modica \cite{MR921549}.

\begin{proposition}
\label{liminf-C2}
Let $\o$ be a bounded open subset of $\R^N$ with boundary of class $C^2$. Given a nonnegative function $c \in L^1(\partial \o)$, $\sigma > 0$, and a continuous function $L \colon \partial \o \to [0,\infty)$ such that $\|L\|_{L^\infty(\partial\o)} \le \sigma$, let $\tau \colon \partial \o \times \R^M \to [- \infty, \infty)$ be a Carath\'eodory function as in \eqref{tau-LB}. Furthermore, let $\f$ and $\h$ be defined as in \eqref{F} and \eqref{H}, respectively. Then, for every $u \in BV(\o; \R^M)$ and every sequence $\{u_n\}_{n \in \N}$ of functions in $W^{1,1}(\o; \R^M)$ such that $u_n \to u$ in $L^1(\o; \R^M)$, we have that 
\[
\liminf_{n \to \infty} \f(u_n) \ge \h(u). 
\]
\end{proposition}

\begin{proof}
Eventually extracting a subsequence (which we do not relabel), we can assume without loss of generality that
\[
\liminf_{n \to \infty} \f(u_n) = \lim_{n \to \infty} \f(u_n) < \infty.
\]
Then, by recalling the definition of $\hat{\tau}$ (see \eqref{eq:sigma}) we deduce that 
\[
\int_{\partial \o}\tau(x, \tr u_n(x)) - \hat{\tau}(x, \tr u(x))\,d\hno \ge - \sigma \int_{\partial \o}|\tr u_n(x) - \tr u(x)|\,d\hno,
\]
and we therefore obtain
\begin{align*}
\f(u_n) - \h(u) & = \sigma|Du_n|(\o) - \sigma|Du|(\o) + \int_{\partial \o}\tau(x, \tr u_n(x)) - \hat{\tau}(x, \tr u(x))\,d\hno \\
& \ge \sigma \left(|Du_n|(\o) - |Du|(\o) - \int_{\partial \o}|\tr u_n(x) - \tr u(x)|\,d\hno\right).
\end{align*}
In turn, to prove the desired inequality it is enough to show that 
\begin{equation}
\label{wtsC2}
\liminf_{n \to \infty} \mathcal{G}_n \ge 0, 
\end{equation}
where
\begin{equation}
\label{Gn}
\mathcal{G}_n \coloneqq |Du_n|(\o) - |Du|(\o) - \int_{\partial \o}|\tr u_n(x) - \tr u(x)|\,d\hno.
\end{equation}
To this end, for each $\gamma > 0$ we let $\o_{\gamma} \coloneqq \{x \in \o : \dist(x, \partial \o) > \gamma\}$ and we select an open set $U_{\gamma}$ with boundary of class $C^2$ such that $\o_{\gamma} \subset U_{\gamma} \subset \o_{\gamma/2}$. Finally, let $V_{\gamma} \coloneqq \o \setminus \overline{U_{\gamma}}$. Then $\partial \o \subset \partial V_{\gamma}$ and $\partial V_{\gamma}$ is of class $C^2$. Therefore, an application of \Cref{trace-C2} yields
\begin{align*}
\int_{\partial \o}|\tr u_n(x) - \tr u(x)|\,d\hno & \le \int_{\partial V_{\gamma}}|\tr u_n(x) - \tr u(x)|\,d\hno \\
& \le |D(u_n - u)|(V_{\gamma}) + C(V_{\gamma}) \int_{V_{\gamma}}|u_n(x) - u(x)|\,dx. 
\end{align*}
Consequently, we get
\begin{equation}
\label{wtsC2-1}
\mathcal{G}_n \ge |Du_n|(\o) - |Du|(\o) - |D(u_n - u)|(V_{\gamma}) - C(V_{\gamma}) \int_{V_{\gamma}}|u_n(x) - u(x)|\,dx,
\end{equation}
and we furthermore notice that
\begin{align}
\mathcal{I}_n  \coloneqq &\  |Du_n|(\o) - |Du|(\o) - |D(u_n - u)|(V_{\gamma}) \notag \\
 = &\ |Du_n|(\o \setminus \overline{V_{\gamma}}) - |Du|(\o \setminus \overline{V_{\gamma}}) + |Du_n|(V_{\gamma}) - |Du|(\overline{V_{\gamma}} \cap \o) - |D(u_n - u)|(V_{\gamma}) \notag \\
 \ge &\ |Du_n|(\o \setminus \overline{V_{\gamma}}) - |Du|(\o \setminus \overline{V_{\gamma}}) - 2 |Du|(\overline{V_{\gamma}} \cap \o). \label{wtsC2-2}
\end{align}
Thus, combining \eqref{wtsC2-1} and \eqref{wtsC2-2} with the fact that $u_n \to u$ in $L^1(\o; \R^M)$ yields
\begin{equation}
\label{wtsC2-3}
\liminf_{n \to \infty}\mathcal{G}_n \ge \liminf_{n \to \infty} \mathcal{I}_n \ge - 2 |Du|(\overline{V_{\gamma}} \cap \o),
\end{equation}
where in the last step we have used the lower semicontinuity of the total variation with respect to convergence in $L^1(\o \setminus \overline{V_{\gamma}}; \R^M)$. Finally, since $u \in BV(\o; \R^M)$, letting $\gamma \to 0$ in \eqref{wtsC2-3} we arrive at \eqref{wtsC2}. This concludes the proof.
\end{proof}

The reminder of this section is devoted to showing that, under a more stringent assumption on the function $L$ in the lower bound for $\tau$ (see \eqref{tau-LB}), the liminf inequality continues to hold also for domains $\o$ whose boundaries are either of class $C^1$ or almost of class $C^1$. 

\begin{proposition}
\label{liminf-C1pw}
Let $\o$ be a bounded open subset of $\R^N$ and assume that $\partial \o$ is almost of class $C^1$. Given a nonnegative function $c \in L^1(\partial \o)$, $\sigma > 0$, and a continuous function $L \colon \partial \o \to [0,\infty)$ such that $\|L\|_{L^\infty(\partial\o)} \le \sigma$, let $\tau \colon \partial \o \times \R^M \to [- \infty, \infty)$ be a Carath\'eodory function as in \eqref{tau-LB}. Let $\f$ and $\h$ be defined as in \eqref{F} and \eqref{H}, respectively. Furthermore, assume that there exist $\tilde{u} \in W^{1,1}(\o; \R^M)$ as in \eqref{F<infty} and $\e_0 > 0$ such that for all $x \in \partial \o$
\[
L(x)q_{\partial \o}(x) \le (1 - 2 \e_0) \sigma,
\]
where $q_{\partial \o}$ is given as in \Cref{q-def}. Then, for every $u \in L^1(\o; \R^M)$ and every sequence $\{u_n\}_{n \in \N}$ of functions in $W^{1,1}(\o; \R^M)$ such that $u_n \to u$ in $L^1(\o; \R^M)$, we have that 
\[
\liminf_{n \to \infty} \f(u_n) \ge \h(u). 
\]
\end{proposition}

We begin by proving two preliminary lemmas. 
\begin{lemma}
\label{lem:cmpt}
Under the assumptions of \Cref{liminf-C1pw}, let $\{u_n\}_{n \in \N} \subset W^{1,1}(\o;\R^M)$ be such that $u_n \to u$ in $L^1(\o;\R^M)$ and
\[
\liminf_{n \to \infty} \f(u_n) < \infty.
\]
Then $u \in BV(\o;\R^M)$ and $\{u_n\}_{n \in \N}$ admits a subsequence, namely $\{u_{n_k}\}_{k \in \N}$, which is bounded in $W^{1,1}(\o; \R^M)$. In particular, $u_{n_k}$ converges to $u$ weakly-$*$ in $BV(\o; \R^M)$.
\end{lemma}

\begin{proof}
Eventually extracting a subsequence (which we do not relabel), we can assume that 
\[
K \coloneqq \liminf_{n \to \infty} \f(u_n) = \lim_{n \to \infty} \f(u_n).
\]
Notice that by \Cref{local-AG} we have
\begin{equation}
\label{bdd-grad}
\int_{\partial \o} L(x)|\tr u_n(x)|\,d\hno \le (1 - \e_0)\sigma \int_{\o}|\nabla u_n(x)|\,dx + \sigma C(\o, \e_0, L/\sigma) \int_{\o}|u_n(x)|\,dx.
\end{equation}
Consequently, using \eqref{tau-LB} and \eqref{bdd-grad} we see that for every $n$ sufficiently large
\begin{align*}
K + 1 \ge \f(u_n) & = \sigma \int_{\o}|\nabla u_n(x)|\,dx + \int_{\partial \o} \tau(x, \tr u_n(x))\, d\hno   \\
& \ge \sigma \int_{\o}|\nabla u_n(x)|\,dx - \|c\|_{L^1(\partial \o)} - \int_{\partial \o} L(x)|\tr u_n(x)| \,d\hno  \\
& \ge \sigma \e_0 \int_{\o}|\nabla u_n(x)|\,dx - \|c\|_{L^1(\partial \o)} - \sigma C(\o,\e_0, L/\sigma) \int_{\o} |u_n(x)| \,dx.
\end{align*}
Since $u_n \to u$ in $L^1(\o; \R^M)$, we readily deduce that for some $\tilde{n} \in \N$
\[
\sup \left\{ \int_\o |\nabla u_n(x)|\,dx : n \ge \tilde{n} \right\} < \infty.
\]
Hence $\{u_n\}_{n \ge \tilde{n}}$ is bounded in $W^{1,1}(\o; \R^M)$. The rest of the proof follows by standard arguments (see, for example, Proposition 3.13 in \cite{AFP}).
\end{proof}

\begin{lemma}
\label{Hfinite}
Let $\o$ be a bounded open subset of $\R^N$ with Lipschitz continuous boundary. Given a nonnegative function $c \in L^1(\partial \o)$, $\sigma > 0$, and a continuous function $L \colon \partial \o \to [0,\infty)$ such that $\|L\|_{L^\infty(\partial\o)} \le \sigma$, let $\tau \colon \partial \o \times \R^M \to [- \infty, \infty)$ be a Carath\'eodory function as in \eqref{tau-LB}. Let $\f$ and $\h$ be defined as in \eqref{F} and \eqref{H}, respectively. Furthermore, assume that there exists $\tilde{u} \in W^{1,1}(\o; \R^M)$ such that $\f(\tilde{u}) < \infty$. Then we have:
\begin{itemize}
\item[$(i)$] $\tau(\cdot, \tr \tilde{u}(\cdot)) \in L^1(\partial \o)$;
\item[$(ii)$] $\hat{\tau}(\cdot, v(\cdot)) \in L^1(\partial \o)$ for all $v \in L^1(\partial \o; \R^M)$ and in particular $\h(u)$ is finite for all $u \in BV(\o; \R^M)$.
\end{itemize}
\end{lemma}
\begin{proof}
Denote with $\tau^+$ and $\tau^-$ the positive and the negative part of $\tau$, respectively, so that 
\[
\tau(x, p) = \tau^+(x, p) - \tau^-(x, p) \qquad  \text{ and } \qquad |\tau(x, p)| = \tau^+(x, p) + \tau^-(x, p).
\]
Then it follows from \eqref{tau-LB} that $\tau^-(\cdot, v(\cdot)) \in L^1(\partial \o)$ for all $v \in L^1(\partial \o; \R^M)$, and moreover, recalling that by assumption $\f(\tilde{u}) < \infty$, we have
\begin{align*}
\infty & > \int_{\partial \o} \tau(x, \tr \tilde{u}(x))\,d\hno \\
& = \int_{\partial \o} \tau^+(x, \tr \tilde{u}(x)) - \tau^-(x, \tr \tilde{u}(x))\,d\hno \\
& \ge \int_{\partial \o} \tau^+(x, \tr \tilde{u}(x))\,d\hno - \|c\|_{L^1(\partial \o)} - \sigma \int_{\partial \o}|\tr \tilde{u}(x)|\,d\hno.
\end{align*}
Since this shows that $\tau^+(\cdot, \tr \tilde{u}(\cdot)) \in L^1(\partial \o)$, statement $(i)$ readily follows. 

On the other hand, since
\[
|\hat{\tau}(x, p)| \le \max\{c(x) + \sigma|p|, |\tau(x, p)|\}
\]
for $\hno$-a.e.\@ $x \in \partial \o$ and for all $p \in \R^M$, we obtain that also $\hat{\tau}(\cdot, \tr \tilde{u}(\cdot)) \in L^1(\partial \o)$. Therefore, for all $v \in L^1(\partial \o; \R^M)$ we have that
\begin{align}
\int_{\partial \o} |\hat{\tau}(x, v(x))|\,d\hno & \le \int_{\partial \o}|\hat{\tau}(x, v(x)) - \hat{\tau}(x, \tr \tilde{u}(x))| + |\hat{\tau}(x, \tr \tilde{u}(x))|\,d\hno \notag \\
& \le \int_{\partial \o}\sigma|v(x) - \tr \tilde{u}(x)| + |\hat{\tau}(x, \tr \tilde{u}(x))|\,d\hno < \infty. \label{hatL1}
\end{align}
This concludes the proof.
\end{proof}

\begin{proof}[Proof of \Cref{liminf-C1pw}]
Notice that by eventually extracting a subsequence (which we do not relabel), we can assume without loss of generality that
\[
\liminf_{n \to \infty} \f(u_n) = \lim_{n \to \infty} \f(u_n) < \infty.
\]
Then, in view of \Cref{lem:cmpt}, we have that $u \in BV(\o; \R^M)$ and furthermore we can find a positive number $\Lambda$ such that
\begin{equation}
\label{Lambda-bound}
\int_{\o}|\nabla u_n(x)|\,dx \le \Lambda
\end{equation}
for every $n$ large enough. We divide the rest of the proof into two steps.
\newline
\textbf{Step 1:} Assume first that $\partial \o$ is of class $C^1$, so that $Q_{\partial \o} = 1$ (see \Cref{rem:Q} and \eqref{Q-def}). Reasoning as in the proof of \Cref{liminf-C2}, our aim is to show that $\liminf_{n \to \infty}\mathcal{G}_n \ge 0$, where $\mathcal{G}_n$ is defined as in \eqref{Gn}. To this end, as in the previous case, we proceed to find an open subset of $\o \setminus \o_{\gamma}$ with boundary of class $C^1$, namely $V_{\gamma}$, such that $\partial \o \subset \partial V_{\gamma}$. Then, since we are no longer in a position to apply \Cref{trace-C2}, we deduce from \Cref{thm:trace} (see also \Cref{rem:Q} and \eqref{Q-def}) with $Q_{\partial V_{\gamma}} = 1$ that for every $\e > 0$ it holds
\[
\mathcal{G}_n \ge |Du_n|(\o) - |Du|(\o) - (1 + \e)|D(u_n - u)|(V_{\gamma}) - C(V_{\gamma}, \e) \int_{V_{\gamma}}|u_n(x) - u(x)|\,dx.
\]
In turn, for every $n$ sufficiently large we obtain
\begin{align*}
\tilde{\mathcal{I}}_n \coloneqq &\  |Du_n|(\o) - |Du|(\o) - (1 + \e)|D(u_n - u)|(V_{\gamma}) \\
\ge &\ |Du_n|(\o \setminus \overline{V_{\gamma}}) - |Du|(\o \setminus \overline{V_{\gamma}}) - \e \Lambda - (2 + \e) |Du|(\overline{V_{\gamma}} \cap \o), 
\end{align*}
where $\Lambda$ is given as in \eqref{Lambda-bound}. Therefore, combining the previous inequalities we get
\[
\liminf_{n \to \infty} \mathcal{G}_n \ge \liminf_{n \to \infty} \tilde{\mathcal{I}}_n \ge - \e \Lambda - (2 + \e) |Du|(\overline{V_{\gamma}} \cap \o).
\]
Finally, letting $\e, \gamma \to 0$ concludes the proof in this case.
\newline
\textbf{Step 2:} In this step we address the general case of a domain with boundary almost of class $C^1$, in the sense of \Cref{a-C1}. Fix $\e, \gamma > 0$. Reasoning as in the proof of \Cref{giusti-C1}, we can find two finite collections of points, namely $\{y^1, \dots, y^{\ell}\} \subset \mathcal{S}$ and $\{z^1, \dots, z^k\} \subset \partial \o \setminus \mathcal{S}$, with the following properties:
\begin{enumerate}[label=(P.\arabic*), ref=P.\arabic*]
\item \label{P1} there are positive constants $\gamma_1, \dots, \gamma_{\ell}$, with $\gamma_i < \gamma$, such that
\begin{equation}
\label{Bg-0}
\sum_{i = 1}^{\ell} \gamma_i^{N - 1} < \gamma
\end{equation}
and
\begin{equation}
\label{Bg-def}
\mathcal{S} \subset \mathcal{B}_{\gamma} \coloneqq \bigcup_{i = 1}^{\ell} B(y^i, \gamma_i);
\end{equation}
\item \label{P2} for each $j \in \{1, \dots, k\}$ there exists an open neighborhood of $z^j$, namely $V_j$, such that $V_j \subset \{x \in \R^N : \dist(x, \partial \o) < \gamma\}$ and (up to a rotation) $\partial \o \cap V_j$ coincides with the graph of a function of class $C^1$ with gradient uniformly bounded from above by $\e$. 
\item \label{P3} the collection of sets in \eqref{P1} and \eqref{P2} forms a covering of $\partial \o$, that is
\[
\partial \o \subset \mathcal{B}_{\gamma} \cup \bigcup_{j = 1}^k V_j.
\]
\end{enumerate}
In view of \eqref{P3}, there exists an open set $U_{\gamma} \subset \o$ such that 
\[
\overline{\o} \subset U_{\gamma} \cup \mathcal{B}_{\gamma} \cup \bigcup_{j = 1}^k V_j.
\]
Let $\{\psi^{\gamma}, \{\psi_i^{\mathcal{S}}\}_{i \le \ell}, \{\psi_j\}_{j \le k}\}$ be a partition of unity on $\overline{\o}$ subordinated to the open covering $\{U_{\gamma}, \{B(y^i, \gamma_i)\}_{i \le \ell}, \{V_j\}_{j \le k}\}$ and denote by $\mathcal{E}$ the excess contact energy, that is 
\begin{equation}
\label{Bun}
\mathcal{E}(u_n, u) \coloneqq \int_{\partial \o} \tau(x, \tr u_n(x)) - \hat{\tau}(x, \tr u(x)) \,d\hno.
\end{equation}
With these notations at hand, we can write
\begin{multline}
\label{tau-hat}
\mathcal{E}(u_n, u) \ge - \sigma \sum_{j = 1}^k \int_{\partial \o} \psi_j(x)|\tr u_n(x) - \tr u(x)|\,d\hno \\ 
+ \sum_{i = 1}^{\ell}\int_{\partial \o} \psi_i^{\mathcal{S}}(x) \left[\tau(x, \tr u_n(x)) - \hat{\tau}(x, \tr u(x)) \right]\,d\hno , 
\end{multline}
and proceed to estimate each term on the right-hand side separately. To begin, we claim that for each $j \in \{1, \dots, k\}$
\begin{multline}
\int_{\partial \o} \psi_j(x)|\tr u_n(x) - \tr u(x)| \,d\hno \le (1 + \e) \biggr(\int_{\o} \psi_j(x) |\nabla u_n(x)|\,dx + \int_{\o} \psi_j(x)\,d|Du|(x) \\
+ \|\nabla \psi_j\|_{L^{\infty}(V_j; \R^N)}\int_{\o}|u_n(x) - u(x)|\,dx\biggr). \label{bdry-bound-j}
\end{multline}
Notice indeed that in view of \eqref{P2}, an application of \Cref{local-trace} yields
\[
\int_{\partial \o} \psi_j(x)|\tr u_n(x) - \tr u(x)|\,d\hno \le (1 + \e) |D(\psi_j(u_n - u))|(\o).
\]
Moreover, since by Leibniz's formula 
\[
D(\psi_j(u_n - u)) = \psi_jD(u_n - u) + (u_n - u) \otimes \nabla \psi_i \mathcal{L}^N,
\] 
we see that
\begin{multline}
\label{tot-var-est}
|D(\psi_j(u_n - u))|(\o) \le \int_{\o} \psi_j(x) |\nabla u_n(x)|\,dx + \int_{\o} \psi_j(x)\,d|Du|(x) \\ + \|\nabla \psi_j\|_{L^{\infty}(V_j; \R^N)}\int_{\o}|u_n(x) - u(x)|\,dx,  
\end{multline}
and \eqref{bdry-bound-j} readily follows. Similarly, we claim that for every $i \in \{1, \dots, \ell \}$ we have
\begin{align}
\label{bdry-bound-S}
\int_{\partial \o} L(x)\psi_i^{\mathcal{S}}(x)|\tr u_n(x) - \tr u(x)|\,d\hno & \le (1 - \e_0) \sigma \int_{\o} \psi_i^{\mathcal{S}}(x) |\nabla u_n(x)|\,dx \notag \\
& \qquad + (1 - \e_0) \sigma \int_{\o} \psi_i^{\mathcal{S}}(x)\,d|Du|(x) \notag \\
& \qquad + \tilde{C}(\o, \e_0, L/\sigma, \psi_i^{\mathcal{S}}) \int_{\o} |u_n(x) - u(x)|\,dx,
\end{align}
where 
\[
\tilde{C}(\o, \e_0, L/\sigma, \psi_i^{\mathcal{S}}) \coloneqq \sigma C(\o, \e_0, L/\sigma) + \|\nabla \psi_i^{\mathcal{S}}\|_{L^{\infty}(B(y^i, \gamma_i); \R^N)}.
\]
Indeed, an application of \Cref{local-AG} yields
\begin{multline}
\label{bdry-bound-0}
\int_{\partial \o} L(x)\psi_i^{\mathcal{S}}(x)|\tr u_n(x) - \tr u(x)|\,d\hno \le (1 - \e_0) \sigma |D(\psi_i^{\mathcal{S}}(u_n - u))|(\o) \\
+ \sigma C(\o, \e_0, L/\sigma)\int_{\o}|u_n(x) - u(x)|\,dx,
\end{multline}
and \eqref{bdry-bound-S} follows from \eqref{bdry-bound-0} with similar computations to those in \eqref{tot-var-est}. Using the fact that for every such $i$
\begin{align*}
\int_{\partial \o} \psi_i^{\mathcal{S}}(x)\tau(x, &\,\tr u_n(x))\,d\hno  \\
& \ge \int_{\partial \o} \psi_i^{\mathcal{S}}(x)( - c(x) - L(x)|\tr u_n(x)|)\,d\hno \\
& \ge \int_{\partial \o} \psi_i^{\mathcal{S}}(x)( - c(x) - L(x)|\tr u_n(x) - \tr u(x)| - L(x)|\tr u(x)|)\,d\hno,
\end{align*}
together with \eqref{tau-hat}, \eqref{bdry-bound-j}, and \eqref{bdry-bound-S}, we deduce that
\begin{align}
\label{tau-diff}
\mathcal{E}(u_n, u) & \ge - (1 + \e)\sigma \left( \int_{\o} (1 - \psi^{\gamma}(x))|\nabla u_n(x)|\,dx + \int_{\o}(1 - \psi^{\gamma}(x))\,d|Du|(x)\right) + \mathcal{Q}(n, \gamma),
\end{align}
where 
\begin{multline}
\mathcal{Q}(n, \gamma) \coloneqq - \biggr((1 + \e) \sigma \sum_{j = 1}^k\|\nabla \psi_j\|_{L^{\infty}(V_j; \R^N)} + \sum_{i = 1}^{\ell}\tilde{C}(\o, \e_0, L/\sigma, \psi_i^{\mathcal{S}})\biggr) \int_{\o}|u_n(x) - u(x)|\,dx \\
- \sum_{i = 1}^{\ell}\int_{\partial \o} \psi_i^{\mathcal{S}}(x)(c(x) + L(x)|\tr u(x)| + \hat{\tau}(x, \tr u(x)))\,d\hno. \label{Qn-def}
\end{multline}
Consequently, it follows from \eqref{tau-diff} that 
\begin{align}
\label{wts-diff}
\f(u_n) - \h(u) & = \sigma \int_{\o} |\nabla u_n(x)|\,dx - \sigma|Du|(\o) + \mathcal{E}(u_n, u) \notag \\
& \ge (1 + \e) \sigma \int_{\o} \psi^{\gamma}(x) |\nabla u_n(x)|\,dx - \sigma |Du|(\o) - \sigma \e \int_{\o} |\nabla u_n(x)|\,dx \notag \\
& \qquad - (1 + \e) \sigma \int_{\o}(1 - \psi^{\gamma}(x))\,d|Du|(x) + \mathcal{Q}(n , \gamma).
\end{align}
Using the fact that, in view of \eqref{P1} and \eqref{P2}, $\psi^{\gamma} = 1$ in $\o_{\gamma} \coloneqq \{x \in \o : \dist(x, \partial \o) > \gamma \}$, from \eqref{wts-diff} we obtain that
\begin{multline}
\label{wts-diff-22}
\f(u_n) - \h(u) \ge \sigma \left(\int_{\o_{\gamma}}|\nabla u_n(x)|\,dx - |Du|(\o_{\gamma})\right) - \sigma \e \int_{\o} |\nabla u_n(x)|\,dx \\
 - (2 + \e)\sigma |Du|(\o \setminus \o_{\gamma}) + \mathcal{Q}(n, \gamma).
\end{multline}
Consequently, by \eqref{Lambda-bound} and \eqref{wts-diff-22}, together with the lower semicontinuity of the total variation with respect to convergence in $L^1(\o_{\gamma}; \R^M)$, we obtain
\begin{equation}
\label{inf-wts}
\liminf_{n \to \infty} \f(u_n) - \h(u) \ge - \sigma \e \Lambda - (2 + \e)\sigma |Du|(\o \setminus \o_{\gamma}) + \liminf_{n \to \infty} \mathcal{Q}(n, \gamma).
\end{equation}
Since $u_n \to u$ in $L^1(\o; \R^M)$, from \eqref{Bg-def} and \eqref{Qn-def} we see that 
\begin{align}
\label{inf-wts-2}
\liminf_{n \to \infty} \mathcal{Q}(n, \gamma) & = - \sum_{i = 1}^{\ell}\int_{\partial \o} \psi_i^{\mathcal{S}}(x)(c(x) + L(x)|\tr u(x)| + \hat{\tau}(x, \tr u(x)))\,d\hno \notag \\
& \ge - \int_{\partial \o \cap \mathcal{B}_{\gamma}} c(x) + \sigma|\tr u(x)| + \hat{\tau}(x, \tr u(x))\,d\hno.
\end{align}
Notice that by \Cref{Hfinite} and the fact that $\hno(\partial \o \cap \mathcal{B}_{\gamma}) \to 0$ as $\gamma \to 0$ (see \eqref{Bg-0}), an application of Lebesgue's dominated convergence theorem readily yields
\[
\lim_{\gamma \to 0}\int_{\partial \o \cap \mathcal{B}_{\gamma}} c(x) + \sigma|\tr u(x)| + \hat{\tau}(x, \tr u(x))\,d\hno = 0.
\]
Thus, since $u \in BV(\o; \R^M)$, the desired result follows from \eqref{inf-wts} and \eqref{inf-wts-2} by letting $\e, \gamma \to 0$. This concludes the proof. 
\end{proof}


\section{Limsup Inequality}
\label{sec-limsup}
To conclude the proof of \Cref{thm:main}, we are left to show the existence of a recovery sequence. Unlike for the liminf inequality, in this case there is no difference in the proof for sets with $C^2$ or almost $C^1$ boundary. This is achieved by the following proposition.

\begin{proposition}
\label{limsup-prop}
Let $\o$ be a bounded open subset of $\R^N$ and assume that $\partial \o$ is almost of class $C^1$. Given a nonnegative function $c \in L^1(\partial \o)$, $\sigma > 0$, and a continuous function $L \colon \partial \o \to [0,\infty)$ such that $\|L\|_{L^\infty(\partial\o)} \le \sigma$, let $\tau \colon \partial \o \times \R^M \to [- \infty, \infty)$ be a Carath\'eodory function as in \eqref{tau-LB}. Let $\f$ and $\h$ be defined as in \eqref{F} and \eqref{H}, respectively. Furthermore, assume that there exist $\tilde{u} \in W^{1,1}(\o; \R^M)$ as in \eqref{F<infty}. Then, for every $u \in L^1(\o; \R^M)$ there exists a sequence $\{u_n\}_{n \in \N}$ of functions in $W^{1,1}(\o; \R^M)$ such that $u_n \to u$ in $L^1(\o; \R^M)$ and 
\[
\limsup_{n \to \infty} \f(u_n) \le \h(u). 
\]
\end{proposition}

Before proceeding with the proof of the limsup inequality, we show that the desired result holds whenever it holds under certain additional assumptions.

\begin{lemma}
\label{simplifications}
For the purpose of proving \Cref{limsup-prop}, it is not restrictive to assume that:
\begin{itemize}
\item[$(i)$] $u \in W^{1,1}(\o; \R^M) \cap C(\overline{\o}; \R^M)$;
\item[$(ii)$] $\|L\|_{L^{\infty}(\partial \o)} < \sigma$.
\end{itemize}
\end{lemma}
Note that the second item in the lemma plays a role only in the case of sets of class $C^2$, for which we only assume $\|L\|_{L^{\infty}(\partial \o)} \le \sigma$ (see \Cref{thm:main}).

\begin{proof}[Proof of \Cref{simplifications}]
We divide the proof into two steps.
\newline
\textbf{Step 1:} Without loss of generality we can assume that $\h(u) < \infty$, since otherwise there is nothing to do. Thus $u \in BV(\o; \R^M)$, and the proof of statement $(i)$ follows by a diagonal argument; we present here the details for the reader's convenience. Fix $u \in BV(\o; \R^M)$ and let $\{v_n\}_{n \in \N}$ be given as in \Cref{smoothapprox}. Then, for every $\e > 0$ there exists $n(\e) \in \N$ such that for every $n \ge n(\e)$ we have 
\begin{equation}
\label{vnlimsup}
\mathcal{H}(v_n) \le \mathcal{H}(u) + \e \quad \text{ and } \quad \int_{\o}|v_n(x) - u(x)|\,dx \le \e.
\end{equation}
For every $n \in \N$, let $\{v_{n, k}\}_{k \in \N} \subset W^{1,1}(\o; \R^M)$ be a recovery sequence for $v_n$, i.e.\@, $v_{n, k} \to v_n$ in $L^1(\o;\R^M)$ as $k \to \infty$ and furthermore
\[
\limsup_{k \to \infty} \f(v_{n,k}) \le \mathcal{H}(v_n).
\]
Then, for every $n \in \N$ we can find $k_n \in \N$ with the property that 
\begin{equation}
\label{vnklimsup}
\f\left(v_{n, k_n}\right) \le \mathcal{H}(v_n) + \frac{1}{n} \quad \text{ and } \quad \int_{\o}|v_{n,k_n}(x) - v_n(x)|\,dx \le \frac{1}{n}.
\end{equation}
Set $u_n \coloneqq v_{n, k_n}$ and observe that by \eqref{vnlimsup} and \eqref{vnklimsup} we have
\[
\int_{\o}|u_n(x) - u(x)|\,dx \le \int_{\o}|v_{n, k_n}(x) - v_n(x)|\,dx + \int_{\o}|v_n(x) - u(x)|\,dx  \le \frac{1}{n} + \e
\]
whenever $n \ge n(\e)$. In particular, by letting first $n \to \infty$ and then $\e \to 0$, we see that $u_n \to u$ in $L^1(\o;\R^M)$. Similarly, since 
\[
\f(u_n) \le \mathcal{H}(v_n) + \frac{1}{n} \le \mathcal{H}(u) + \e + \frac{1}{n},
\]
we readily obtain that 
\[
\limsup_{n \to \infty} \f(u_n) \le \mathcal{H}(u) + \e.
\]
Once again, letting $\e \to 0$ yields the desired result.
\newline
\textbf{Step 2:} For $k \in \N$, we let $\tau_k \colon \partial \o \times \R^M \to \R$ be defined via 
\[
\tau_k(x, p) \coloneqq \tau(x, p) + \frac{1}{k}|p|.
\]
As one can readily check, $\tau_k$ satisfies a similar lower bound to \eqref{tau-LB}, with $L$ replaced by 
\[
L_k(x) \coloneqq L(x) - \frac{1}{k}.
\]
Since $\|L_k\|_{L^{\infty}(\partial \o)} < \sigma$, we have that for every $u \in BV(\o; \R^M)$ there exists a sequence $\{u^k_n\}_{n \in \N}$ of functions in $W^{1,1}(\o; \R^M)$ such that $u_n^k \to u$ in $L^1(\o; \R^M)$ as $n \to \infty$, and 
\[
\limsup_{n \to \infty} \f_k(u^k_n) \le \sigma|Du|(\o) + \int_{\partial \o}\hat{\tau}_k(x, u(x))\,d\hno,
\]
where $\f_k$ is the functional obtained by replacing $\tau$ with $\tau_k$ in the definition of $\f$ (see \eqref{F}). Similarly, for every $x \in \partial \o$, $\hat{\tau}_k(x, \cdot)$ is defined as the $\sigma$-Yosida transform of $\tau_k$ (see \eqref{eq:sigma} for the precise definition). We claim that $\hat{\tau}_k(x, u(x)) \to \hat{\tau}(x, u(x))$ for $\hno$-a.e.\@ $x \in \partial \o$. Indeed, if we fix $\delta > 0$ and let $p_x$ be such that 
\[
\hat{\tau}(x, u(x)) + \delta \ge \tau(x, p_x) + \sigma|p_x - u(x)|,
\]
then, for $\hno$-a.e.\@ $x \in \partial \o$ we have
\begin{equation}
\label{ptw-k}
\hat{\tau}(x, u(x)) \le \hat{\tau}_k(x, u(x)) \le \tau_k(x, p_x) + \sigma|p_x - u(x)| \le \hat{\tau}(x, u(x)) + \frac{1}{k}|p_x| + \delta.
\end{equation}
The claim readily follows by letting first $k \to \infty$ and then $\delta \to 0$. Let $\tilde{u}$ be given as in \eqref{F<infty}; then, for $\hno$-a.e.\@ $x \in \partial \o$ we have that 
\begin{align*}
- c(x) - \sigma|u(x)| \le \hat{\tau}_k(x, u(x)) & \le \tau_k(x, \tr \tilde{u}(x)) + \sigma |u(x) - \tr \tilde{u}(x)|   \\
& = \tau(x, \tr \tilde{u}(x)) + \frac{1}{k}|\tr \tilde{u}(x)| + \sigma|u(x) - \tr \tilde{u}(x)|,
\end{align*}
and therefore there exists a function $g \in L^1(\partial \o)$ such that $|\hat{\tau}_k(x, u(x))| \le g(x)$ holds for $\hno$-a.e.\@ $x \in \partial \o$. In turn, by Lebesgue's dominated convergence theorem we have that
\[
\lim_{k \to \infty} \int_{\partial \o} \hat{\tau}_k(x, u(x)) \,d\hno = \int_{\partial \o}\hat{\tau}(x, u(x))\,d\hno.
\]
Since $\f \le \f_k$, the rest of the proof follows by a standard diagonal argument, as in the previous step; we omit the details.
\end{proof}

\begin{proof}[Proof of \Cref{limsup-prop}]
In view of \Cref{simplifications}, throughout the proof we can assume that $\|L\|_{L^{\infty}(\partial \o)} < \sigma$ and furthermore, we fix a function $u \in W^{1,1}(\o; \R^M) \cap C(\overline{\o}; \R^M)$. We divide the proof into two steps.
\newline
\textbf{Step 1:} Let $\hat{\tau}$ be given as in \eqref{eq:sigma} and notice that since both $\tau$ and $\hat{\tau}$ are Carath\'eodory functions, by \Cref{scorza-dragoni} we have that for every $\eta > 0$ there exists a compact set $C_{\eta} \subset \partial \o$, with
\[
\hno(\partial \o \setminus C_{\eta}) \le \eta,
\]
such that the restrictions of $\tau$ and $\hat{\tau}$ to $C_{\eta} \times \R^M$ are real-valued continuous functions. Assume without loss of generality that $\|c\|_{L^1(\partial \o)} > 0$ and let
\[
\widetilde{C}_{\eta} \coloneqq C_{\eta} \cap \left\{x \in \partial \o : c(x) \le A_{\eta} \coloneqq \eta^{-1}\|c\|_{L^1(\partial \o)}\right\}.
\]
Then, as one can readily check, we have that 
\begin{equation}
\label{smallhno1}
\hno (\partial \o \setminus \widetilde{C}_{\eta} ) \le \hno \left(\partial \o \setminus C_{\eta} \right) + \hno \left( \{x \in \partial \o : c(x) > A_{\eta}\} \right) \le 2\eta. 
\end{equation}
Finally, by the regularity of the Hausdorff measure and in view of \eqref{smallhno1}, we can find a compact set $D_{\eta} \subset \widetilde{C}_{\eta}$ such that 
\begin{equation}
\label{smallhno}
\hno(\partial \o \setminus D_{\eta}) \le 3 \eta
\end{equation}
and with the property that \eqref{tau-LB} holds for every $x \in D_{\eta}$. The purpose of this step is to show that, if $u$ is given as above, for every $\e > 0$ there exists a measurable function $p_{\e, \eta} \colon \partial \o \to \R^M$ with the property that for every $x \in D_{\eta}$ we have
\begin{equation}
\label{approxsigma}
\hat{\tau}(x, u(x)) + \e \ge \tau(x, p_{\e, \eta}(x)) + \sigma |u(x) - p_{\e, \eta}(x)|.
\end{equation}
To prove the claim, notice that if $p \in \R^M$ is such that 
\begin{equation}
\label{pbound}
\hat{\tau}(x, u(x)) + 1 \ge \tau(x, p) + \sigma |u(x) - p|,
\end{equation}
then \eqref{tau-LB} implies that
\begin{align*}
\sigma |p| & \le \sigma |u(x) - p| + \sigma |u| \\
& \le \hat{\tau}(x, u(x)) - \tau(x, p) + 1 + \sigma |u(x)| \\
& \le \hat{\tau}(x, u(x)) + c(x) + L(x)|p| + 1 + \sigma |u(x)|.
\end{align*}
Therefore, for every $x \in D_{\eta}$ we have
\[
\left\{\sigma - \|L\|_{L^{\infty}(\partial \o)}\right\}|p| \le A_{\eta} + 1 + \max\left\{\hat{\tau}(x, u(x)) + \sigma |u(x)| : x \in D_{\eta}\right\}.
\]
In particular, we have shown that there exists a positive number $R_{\eta}$ such that if $x \in D_{\eta}$ and $p$ satisfies \eqref{pbound}, then $|p| \le R_{\eta}$. Next, observe that since $u$ is continuous on the set $D_{\eta}$, so is the map $x \mapsto \hat{\tau}(\cdot , u(\cdot))$; in turn, both maps are uniformly continuous in $D_{\eta}$. Similarly, $\tau$ is uniformly continuous when restricted to the compact set $D_{\eta} \times \overline{B(0,R_{\eta})}$. Thus, for every $\e > 0$ there exists a number $\delta > 0$ such that for every $x \in D_{\eta}$, every $y \in B(x, \delta) \cap D_{\eta}$, and every $p \in \overline{B(0,R_{\eta})}$ we have 
\begin{align}
\sigma |u(x) - u(y)| & \le \frac{\e}{4}, \label{unifcont1} \\
|\hat{\tau}(x, u(x)) - \hat{\tau}(y,u(y))| & \le \frac{\e}{4}, \label{unifcont2}\\
|\tau(x, p) - \tau(y,p)| & \le \frac{\e}{4}. \label{unifcont3}
\end{align}
Once again, form the compactness of $D_{\eta}$ we see that there exists a set of points $\{x^1, \dots, x^k\} \subset D_{\eta}$ such that $D_{\eta}$ is contained in the union of the balls $\{B(x^j, \delta)\}_{j \le k}$. By definition of $\hat{\tau}$, for every $j \in \{1, \dots, k\}$ we can find $p_j \in \overline{B(0,R_{\eta})}$ with the property that 
\begin{equation}
\label{pi-def}
\hat{\tau}(x^j, u(x^j)) + \frac{\e}{4} \ge \tau (x^j, p_j) + \sigma |p_j - u(x^j)|.
\end{equation}
Then, in view of the inequalities in \eqref{unifcont1}, \eqref{unifcont2}, \eqref{unifcont3}, and \eqref{pi-def}, for every $x \in B(x^j, \delta) \cap D_{\eta}$ we have
\begin{align}
\hat{\tau}(x, u(x)) & \ge \hat{\tau}(x^j, u(x^j)) - |\hat{\tau} (x, u(x)) - \hat{\tau} (x^j, u(x^j))| \notag \\
& \ge \tau(x^j, p_j) + \sigma |p_j - u(x^j)| - \frac{\e}{2} \notag  \\
& \ge \tau(x, p_j) - |\tau(x^j, p_j) - \tau(x, p_j)| + \sigma|p_j - u(x)| - \sigma |u(x^j) - u(x)| - \frac{\e}{2} \notag  \\
& \ge \tau(x, p_j) + \sigma|p_j - u(x)| - \e. \label{peINEQ}
\end{align}
Let $p_{\e, \eta}$ be defined via
\[
p_{\e, \eta}(x) \coloneqq 
\left\{
\arraycolsep=5pt\def\arraystretch{1.5}
\begin{array}{ll}
\displaystyle p_1 & \text{ if } x \in D_{\eta} \cap B(x^1, \delta) , \\
\displaystyle p_j & \text{ if } x \in D_{\eta} \cap B(x^j, \delta) \setminus \bigcup_{i = 1}^{j - 1}B(x^i, \delta), \\
\displaystyle \tr \tilde{u} & \text{ if } x \in \partial \o \setminus D_{\eta}, 
\end{array}
\right.
\]
where $\tilde{u}$ is given as in \eqref{F<infty}.
\newline
\textbf{Step 2:} In this step we carry out the construction of an approximate recovery sequence for $u$. Observe that the function $p_{\e, \eta}$ defined in the previous step belongs to the space $L^1(\partial \o; \R^M)$; consequently, by \Cref{IF-lem} we can find a sequence $\{u_n\}_{n \in \N}$ of functions in $W^{1,1}(\o; \R^M)$ with trace $p_{\e, \eta}$, such that $u_n \to u$ in $L^1(\o; \R^M)$ and 
\[
\limsup_{n \to \infty} \int_{\o}|\nabla u_n(x)|\,dx \le \int_{\o}|\nabla u(x)|\,dx + \int_{\partial \o}|u(x) - p_{\e, \eta}(x)|\,d\hno.
\]
In turn, by \eqref{approxsigma} we have that
\begin{align*}
\limsup_{n \to \infty} \f(u_n) & \le  \int_{\o}\sigma |\nabla u(x)|\,dx + \int_{\partial \o}\tau(x, p_{\e, \eta}(x)) + \sigma |u(x) - p_{\e, \eta}(x)|\,d\hno \\
& \le \mathcal{H}(u) + \mathcal{R}(\e, \eta),
\end{align*}
where 
\[
\mathcal{R}(\e, \eta) \coloneqq \e \hno(D_{\eta}) + \int_{\partial \o \setminus D_{\eta}}\sigma |u(x) - \tr \tilde{u}(x)| + \tau(x, \tr \tilde{u}(x)) - \hat{\tau}(x, u(x))\,d\hno. 
\]
As one can readily check, \Cref{Hfinite} and \eqref{smallhno} imply that
\[
\lim_{\e \to 0^+} \lim_{\eta \to 0^+} \mathcal{R}(\e, \eta) = 0.
\]
This concludes the proof.
\end{proof}

\begin{remark}
The first step in the proof of the \Cref{limsup-prop} can be simplified significantly if $\tau$ is independent of $x$. Furthermore, in this case it is enough to require that $\tau$ is only Borel measurable. Indeed, for a given function $u \in W^{1,1}(\o;\R^M) \cap C(\overline{\o}; \R^M)$ and $\e > 0$, our aim in this case is to find an integrable function $p_{\e} \colon \partial \Omega \to \R^M$ with the property that 
\begin{equation}
\label{1no-x}
\hat{\tau}(u(x)) + \e \ge \tau(p_{\e}(x)) + \sigma |u(x) - p_{\e}(x)|
\end{equation}
for $\hno$-a.e.\@ $x \in \partial \o$. Reasoning as above, we can find a positive number $\delta$ such that $\sigma|u(x) - u(y)| \le \e/3$ whenever $|x - y| < \delta$, a finite number of points $\{x^1, \dots, x^k\} \subset \partial \o$ with the property that $\{B(x^j, \delta)\}_{j \le k}$ is a covering of $\partial \o$, and finally points $p_j \in \R^M$ such that 
\[
\hat{\tau}(u(x^j)) + \frac{\e}{3} \ge \tau(p_j) + \sigma |u(x^j) - p_j|.
\]
The condition in \eqref{1no-x} follows from similar computations to \eqref{peINEQ}, provided $p_{\e}$ is opportunely defined to be a simple function taking only values in $\{p_1, \dots, p_k\}$. Furthermore, in this case, the approximation in \Cref{simplifications} $(ii)$ is no longer needed, even when $\o$ is a smooth domain with boundary of class $C^2$.
\end{remark}


\section{Further generalizations and remarks}
\label{sec-extensions}

In this final section, we discuss how to suitably modify the arguments presented above in order to obtain two variants of \Cref{thm:main}. The first and main result of the section concerns the possibility to extend our analysis to include surface densities that are not necessarily continuous in the second variable. Lastly, we examine the case of a general Lipschitz domain under rather strict assumptions on $\tau$.

\subsection{Normal integrands}
\label{sec-normal}
In this subsection we prove that, under a mild integrability condition, we can further relax the regularity assumptions on the surface energy density $\tau$. To be precise, throughout the following we assume that $\tau \colon \partial \o \times \R^M \to [- \infty, \infty)$ satisfies:
\begin{enumerate}[label=(N.\arabic*), ref=N.\arabic*]
\item \label{N1} $\tau$ is $\Sigma(\partial \o, \hno) \times \mathcal{B}(\R^M)$-measurable, where $\Sigma(\partial \o, \hno)$ and $\mathcal{B}(\R^M)$ denote the $\sigma$-algebra of Hausdorff measurable subsets of $\partial \o$ and the Borel $\sigma$-algebra on $\R^M$, respectively;
\item \label{N2} $\tau(x, \cdot)$ is upper semicontinuous for $\hno$-a.e.\@ $x \in \partial \o$.
\end{enumerate}
Let us remark that if in condition $(ii)$ we replace \emph{upper} semicontinuous with \emph{lower} semicontinuous, then the class of functions we obtain is commonly referred to as \emph{normal functions} or \emph{normal integrands} in the relevant literature. This class is particularly well suited for problems in the calculus of variations; we refer to \cite{FL} and \cite{MR0512209} for more information on the subject.

For $\tau$ as above, set
\[
M_j(x) \coloneqq \max \left\{ \sup \left\{\tau(x, p) : p \in \overline{B(0, j)} \right\}, 0\right\}.
\]
Consider a partition of unity, namely $\{\psi_j\}_{j \in \N}$, subordinated to the open covering $\{U_j\}_{j \in \N}$, where the sets $U_j$ are defined via
\[
\begin{array}{ll}
U_1 \coloneqq B(0, 2), & \\
U_j \coloneqq B(0, j + 1) \setminus \overline{B(0, j - 1)} &  \text{ for } j \ge 2,
\end{array}
\]
and define 
\begin{equation}
\label{UBT}
T(x, p) \coloneqq \sum_{j = 1}^{\infty}M_{j + 2}(x)\psi_j(p).
\end{equation}
Notice that $T \colon \partial \o \times \R^M \to [0, \infty)$ is a Carath\'eodory function with the property that $\tau \le T$. 

The main result of this subsection can be stated as follows.
\begin{theorem}
\label{normalthm}
The conclusions of \Cref{thm:main} continue to hold even if $\tau$ fails to be Carath\'eodory, provided that the following conditions are satisfied:
\begin{itemize}
\item[$(i)$] $\tau$ satisfies \eqref{N1} and \eqref{N2}, i.e.\@, $-\tau$ is a normal integrand;
\item[$(ii)$] there exists $u_1 \in L^1(\partial \o; \R^M)$ such that 
\[
T(\cdot, u_1(\cdot)) \in L^1(\partial \o),
\]
where the function $T$ is defined as in \eqref{UBT}.
\end{itemize}
\end{theorem}

We begin by recalling a well-known approximation result, which states that every normal integrand is the pointwise limit of an increasing sequence of Carath\'eodory functions (see, for example, Remark 1 in \cite{MR344980} and Corollary 1 in \cite{MR699706}). A proof of this fact, adapted to our setting, is included here for the reader's convenience. 

\begin{lemma}
\label{Car-approx}
Let $\tau \colon \partial \o \times \R^M \to [- \infty, \infty)$ be as in \eqref{N1} and \eqref{N2}. Then there exists a decreasing sequence of Carath\'eodory functions, namely $\{\tau_k\}_{k \in \N}$, such that $\tau_k(x, \cdot) \to \tau(x, \cdot)$ pointwise in $\R^M$ for $\hno$-a.e.\@ $x \in \partial \o$.
\end{lemma}

\begin{proof}
For $T$ as in \eqref{UBT}, let $t(x, p) \coloneqq \tau(x, p) - T(x, p)$, set 
\[
t_k(x, p) \coloneqq \sup \left\{t(x, q) - k |p - q| : q \in \R^M\right\},
\]
and define
\[
\tau_k(x, p) \coloneqq t_k(x, p) + T(x,p).
\]
We claim that the sequence $\{\tau_k\}_{k \in \N}$ has all the desired properties. Indeed, it follows readily from the definition that $\{\tau_k\}_{k \in \N}$ is a decreasing sequence of Carath\'eodory functions. Thus, to conclude it is enough to show that $\tau_k(x, \cdot) \to \tau(x, \cdot)$ for $\hno$-a.e.\@ $x \in \partial \o$. This, in turn, follows from the fact that $t(x, \cdot)$ is upper semicontinuous for $\hno$-a.e.\@ $x \in \partial \o$. Indeed, for every such $x$ we have that $t_k(x, \cdot) \to t(x, \cdot)$ pointwise in $\R^M$ (see, for example, Lemma 5.30 in \cite{FL}).
\end{proof}

\begin{proof}[Proof of \Cref{normalthm}]
Observe that the proof of the liminf inequality remains unchanged. Thus, in the following we describe how to suitably modify the arguments presented in the previous section. Let $\{\tau_k\}_{k \in \N}$ be given as in \Cref{Car-approx}. Then, reasoning as in the proof of \Cref{simplifications}, it is enough to show that 
\begin{equation}
\label{nts-normal}
\int_{\partial \o} \hat{\tau}_k(x, u(x))\,d\hno \to \int_{\partial \o} \hat{\tau}(x, u(x))\,d\hno 
\end{equation}
for all $u \in L^1(\partial \o; \R^M)$. To this end, let $u \in L^1(\partial \o; \R^M)$ be given. As one can readily check (for example, by reproducing the argument used in \eqref{ptw-k}), $\hat{\tau}_k(\cdot, u(\cdot)) \to \hat{\tau}(\cdot, u(\cdot))$ pointwise almost everywhere in $\partial \o$. Notice that for $u_1$ as in the statement, we have that $\tau_1(\cdot, u_1(\cdot)) \in L^1(\partial \o)$. Indeed, since $t_1 \le 0$, for $\hno$-a.e.\@ $x \in \partial \o$ we have that
\[
- c(x) - \sigma |u_1(x)| \le \tau(x, u_1(x)) \le \tau_1(x, u_1(x)) = t_1(x, u_1(x)) + T(x, u_1(x)) \le T(x, u_1(x)).
\]
In turn, reasoning as in \eqref{hatL1}, we see that $\hat{\tau}_1(\cdot, u(\cdot)) \in L^1(\partial \o)$ for all $u \in L^1(\partial \o; \R^M)$. In particular, since 
\[
|\hat{\tau}_k(x, u(x))| \le \max\left\{c(x) + \sigma |u(x)|, |\hat{\tau}_1(x, u(x))|\right\}
\]
holds for $\hno$-a.e.\@ $x \in \partial \o$, we are in a position to apply Lebesgue's dominated convergence theorem and \eqref{nts-normal} readily follows. This concludes the proof.
\end{proof}


\subsection{Lower semicontinuity on Lipschitz domains}
\label{sec-Lip}
We conclude with a brief discussion of the case in which $\o$ is a general Lipschitz domain. To be precise, we show that the techniques presented in \Cref{sec-liminf} can be readily modified to show that $\f$ admits a lower semicontinuous extension in $BV(\o; \R^M)$. The increased generality on the domain, however, comes with the drawback that we significantly restrict the class of surface densities for which the result applies to those which satisfy a certain Lipschitz condition. This is made precise in the following proposition.
 
\begin{proposition}
\label{Lip-dom}
Let $\o$ be a bounded open subset of $\R^N$ with Lipschitz continuous boundary. Let $\tau \colon \partial \o \times \R^M \to [- \infty, \infty)$ be a Carath\'eodory function such that 
\begin{equation}
\label{well-def-Lip}
\tau(x, p) \ge - c(x) - \beta |p|
\end{equation}
for $\hno$-a.e.\@ $x \in \partial \o$ and for all $p \in \R^M$, where $c \in L^1(\partial \o)$ is nonnegative and $\beta > 0$. Moreover, assume that
\begin{equation}
\label{Lip-lip}
|\tau(x, p) - \tau(x, q)| \le \gamma(x)|p - q|
\end{equation}
for $\hno$-a.e.\@ $x \in \partial \o$ and for all $p, q \in \R^M$, where $\gamma \colon \partial \o \to [0, \infty)$ is a continuous function. Given $\sigma > 0$, let $q_{\partial \o}$ be given as in \Cref{q-def}, and assume that for some $\e_0 > 0$
\begin{equation}
\label{Lip-gamma}
\gamma(x)q_{\partial \o}(x) \le (1 - 2\e_0)\sigma
\end{equation}
for all $x \in \partial \o$. Furthermore, let $\f$ and $\overline{\f}$ be given as in \eqref{F} and \eqref{Fbar}, respectively. Then, for every $u \in BV(\o; \R^M)$ we have that
\[
\overline{\f}(u) = |Du|(\o) + \int_{\partial \o}\tau(x, \tr u(x))\,d\hno.
\]
\end{proposition}
\begin{proof}
The proof follows from a nearly identical (but simpler) argument to that used in \Cref{liminf-C2} and therefore we omit it. 
\end{proof}

\begin{remark}
\label{lip-rem}
Some comments on the assumptions of \Cref{Lip-dom} are in order.
\begin{itemize} 
\item[$(i)$] Condition \eqref{well-def-Lip} is only required so that $\f$ is well-defined. In addition, notice that if in \eqref{well-def-Lip} we could replace the positive constant $\beta$ with a continuous function $L$ as in \eqref{upperbound}, we would then also obtain that $\overline{\f}(u) = \infty$ for all $u \in L^1(\o; \R^M) \setminus BV(\o; \R^M)$ by the same argument used in \Cref{lem:cmpt}. In turn, this condition is readily satisfied if, for example, there exists $q \in \R^M$ such that $\tau(\cdot, q) \in L^{\infty}(\partial \o)$. Indeed, this would imply that for $\hno$-a.e.\@ $x \in \partial \o$ 
\[
\tau(x, p) \ge - \|\tau(\cdot, q)\|_{L^{\infty}(\partial \o)} - \gamma(x) |p - q| \ge - \|\tau(\cdot, q)\|_{L^{\infty}(\partial \o)} - \gamma(x) |q| - \gamma(x) |p|,
\]
and the desired result follows by \eqref{Lip-gamma}.
\item[$(ii)$] An analogous condition to \eqref{Lip-gamma} was previously identified by Giusti (see \cite{MR482506}) in the context of non-parametric surfaces of prescribed mean curvature, where it is customary to assume that $\tau(x, p) \coloneqq \gamma(x)p$, for $p \in \R$.
\item[$(iii)$] While on one hand \Cref{thm:main} already shows that a much finer result holds for more regular domains, it is worth remarking that a Lipschitz assumption of the form \eqref{Lip-gamma} is also not optimal in the case of a general Lipschitz domain. Indeed, in \cite{MR921549} Modica proved that for $\sigma = 1$, if we let $\tau(x,p) \coloneqq |p - \psi(x)|$, where $\psi \in L^1(\partial \o)$, then $\f$ is lower semicontinuous on any Lipschitz domain. 
\end{itemize}
\end{remark}


\subsection*{Acknowledgements}
The research of the first author has been supported by Grant Nos.\@ EP/R013527/1 and EP/R013527/2 ``Designer Microstructure via Optimal Transport Theory" of David Bourne during the period at Heriot-Watt University. The work of the second author was funded by the research support programs of Charles University under Grant Nos.\@ PRIMUS/19/SCI/01 and UNCE/SCI/023, and by the Czech Science Foundation (GA\v{C}R) under Grant No.\@ GJ19-11707Y. The authors would also like to thank Giovanni Leoni for his helpful insights.


\bibliographystyle{siam}
\bibliography{References}

\end{document}